\documentclass[a4paper,reqno,11pt]{amsart}

\usepackage{amsmath, amsfonts, amssymb, amsthm, amscd}
\usepackage{perpage}
\usepackage{mathabx}
\usepackage{tikz}
\usepackage{caption,subcaption}

\usepackage{dsfont} 

\usepackage[a4paper,scale={0.72,0.74},marginratio={1:1},footskip=7mm,headsep=10mm]{geometry}

\usepackage{hyperref}

\makeatletter
\def\@secnumfont{\bfseries\scshape}

\def\section{\@startsection{section}{1}
  \z@{.9\linespacing\@plus\linespacing}{.5\linespacing}%
  {\normalfont\large\bfseries\scshape\centering}}

\def\subsection{\@startsection{subsection}{2}%
  \z@{.5\linespacing\@plus.7\linespacing}{-.5em}%
  {\normalfont\bfseries\scshape}}

\def\@secnumfont{\scshape}

\def\subsubsection{\@startsection{subsubsection}{3}%
  \z@{.5\linespacing\@plus.7\linespacing}{-.5em}%
  {\normalfont\scshape}}

\def\specialsection{\@startsection{section}{1}%
  \z@{\linespacing\@plus\linespacing}{.5\linespacing}%
  {\normalfont\centering\large\bfseries\scshape}}
\makeatother

\makeatletter

\renewenvironment{proof}[1][\proofname]{\par
\pushQED{\qed}%
\normalfont \topsep4\p@\@plus4\p@\relax
\trivlist
\item[\hskip\labelsep
\bfseries
#1\@addpunct{.}]\ignorespaces
}{%
\popQED\endtrivlist\@endpefalse
}
\makeatother

\makeatletter
\newcommand \Dotfill {\leavevmode \leaders \hb@xt@ 6pt{\hss .\hss }\hfill \kern \z@}
\makeatother

\makeatletter
\def\@tocline#1#2#3#4#5#6#7{\relax
  \ifnum #1>\c@tocdepth 
  \else
    \par \addpenalty\@secpenalty\addvspace{#2}%
    \begingroup \hyphenpenalty\@M
    \@ifempty{#4}{%
      \@tempdima\csname r@tocindent\number#1\endcsname\relax
    }{%
      \@tempdima#4\relax
    }%
    \parindent\z@ \leftskip#3\relax \advance\leftskip\@tempdima\relax
    \rightskip\@pnumwidth plus4em \parfillskip-\@pnumwidth
    #5\leavevmode\hskip-\@tempdima
      \ifcase #1
       \or\or \hskip 1.65em \or \hskip 3.3em \else \hskip 4.95em \fi%
      #6\nobreak\relax
    \Dotfill
    \hbox to\@pnumwidth{\@tocpagenum{#7}}\par
    \nobreak
    \endgroup
  \fi}
\makeatother

\makeatletter
\def\l@section{\@tocline{1}{0pt}{1pc}{}{\scshape}}
\renewcommand{\tocsection}[3]{%
\indentlabel{\@ifnotempty{#2}{\ignorespaces#1 #2.\hskip 0.7em}}#3}
\def\l@subsection{\@tocline{2}{0pt}{1pc}{5pc}{}}

\def\l@subsubsection{\@tocline{3}{0pt}{1pc}{7pc}{}}

\makeatother

\numberwithin{equation}{section}

\newtheoremstyle{mytheorem}{.7\linespacing\@plus.3\linespacing}{.7\linespacing\@plus.3\linespacing}%
     {\itshape}
     {}
     {\bfseries}
     {. }
     {0.3ex}
     {\thmname{{\bfseries #1}}\thmnumber{ {\bfseries #2}}\thmnote{ (#3)}}  

\theoremstyle{mytheorem}

\newtheorem{theorem}{Theorem}[section]
\newtheorem{lemma}[theorem]{Lemma}

\newtheorem{remark}[theorem]{Remark}

\newcommand{\bbB}{{\ensuremath{\mathbb B}} }
\newcommand{\bbE}{{\ensuremath{\mathbb E}} }
\newcommand{\bbP}{{\ensuremath{\mathbb P}} }
\newcommand{\bbT}{{\ensuremath{\mathbb T}} }

\newcommand{\cA}{{\ensuremath{\mathcal A}} }
\newcommand{\cB}{{\ensuremath{\mathcal B}} }
\newcommand{\cD}{{\ensuremath{\mathcal D}} }
\newcommand{\cE}{{\ensuremath{\mathcal E}} }
\newcommand{\cN}{{\ensuremath{\mathcal N}} }
\newcommand{\cS}{{\ensuremath{\mathcal S}} }
\newcommand{\dd}{\text{\rm d}}             

\newcommand{\R}{\mathbb{R}}

\newcommand{\Z}{\mathbb{Z}}
\newcommand{\N}{\mathbb{N}}

\newcommand{\T}{\mathbb{T}}

\newcommand{\PEfont}{\mathrm}
\newcommand{\p}{\ensuremath{\PEfont P}}
\newcommand{\E}{\PEfont E}
\renewcommand{\P}{\p}

\DeclareMathOperator{\bbvar}{\ensuremath{\mathbb{V}ar}}
\DeclareMathOperator{\bbcov}{\ensuremath{\mathbb{C}ov}}

\newcommand{\ind}{\mathds{1}}

\newenvironment{myenumerate}{%
\renewcommand{\theenumi}{\arabic{enumi}}%
\renewcommand{\labelenumi}{{\rm(\theenumi)}}%
\begin{list}{\labelenumi}
	{%
	\setlength{\itemsep}{0.4em}%
	\setlength{\topsep}{0.5em}%
	\setlength\leftmargin{2.45em}%
	\setlength\labelwidth{2.05em}%
	\setlength{\labelsep}{0.4em}%
	\usecounter{enumi}%
	}%
	}%
{\end{list}
}

{\end{list}
}

{\end{list}
}

\renewenvironment{enumerate}{
\begin{myenumerate}}%
{\end{myenumerate}}

\newenvironment{myitemize}{%
\begin{list}{$\bullet$}%
 	{%
	\setlength{\itemsep}{0.4em}%
	\setlength{\topsep}{0.5em}%
	\setlength\leftmargin{2.65em}%
	\setlength\labelwidth{2.65em}%
	\setlength{\labelsep}{0.4em}%
	}%
	}%
{\end{list}}

\renewenvironment{itemize}{
\begin{myitemize}}%
{\end{myitemize}}

\MakePerPage[2]{footnote} 

\newcommand{\rme}{\mathrm{e}}
\newcommand{\dom}{\mathrm{dom}}
\newcommand{\cg}{\mathrm{cg}}
\newcommand{\diff}{\mathrm{diff}}

\title{Gaussian Limits for Subcritical Chaos}

\author{Francesco Caravenna, Francesca Cottini}

\address{Dipartimento di Matematica e Applicazioni\\
 Universit\`a degli Studi di Milano-Bicocca\\
 via Cozzi 55, 20125 Milano, Italy}

\email{francesco.caravenna@unimib.it, francesca.cottini@unimib.it}

\begin{document}

\begin{abstract}
We present a simple criterion, only based on second moment assumptions,
for the convergence of polynomial or Wiener chaos to a Gaussian limit.
We exploit this criterion to obtain new Gaussian asymptotics for 
the partition functions of two-dimensional directed polymers in the sub-critical regime,
including a singular product between the partition function and the disorder.
These results can also be applied to the KPZ and Stochastic Heat Equation.
As a tool of independent interest, we derive an explicit
chaos expansion which sharply approximates
the \emph{logarithm} of the partition function.
\end{abstract}


\maketitle

\section{Introduction}

In this paper we investigate the convergence to a Gaussian limit for 
 random variables that have the structure of a \emph{polynomial chaos},
that is a multi-linear polynomial of independent random variables, or alternatively of
a \emph{Wiener chaos}, that is a sum of multiple Wiener integrals with respect to a
Gaussian random measure.
Our main motivation is the study of directed polymers in random environment,
whose partition function provides a discretization of the solution of the 
multiplicative Stochastic Heat Equation (SHE), while its logarithm corresponds
to the solution of the  KPZ equation.
Many convergence results to Gaussian limits have been obtained
in recent years for directed polymers and for SHE and KPZ
(see the discussion in Section~\ref{sec:applications})
based on polynomial chaos or Wiener chaos,
often exploiting the Fourth Moment Theorem and variations thereof.
Our purpose is to present a general approach which makes it possible to
recover these results in a simpler and unified way and, furthermore, to obtain
novel results. Let us give an overview of the paper.

\smallskip

In Section~\ref{sec:convergence} we state our first main result: a general criterion
for the convergence of polynomial chaos or Wiener chaos to a Gaussian limit 
\emph{only based on second moment assumptions},
see Theorems~\ref{th:main} and~\ref{th:maincont}.
Besides the fact that we do not require higher moment bounds, 
we can work directly with a superposition of chaos of different orders,
with no need of treating them individually as in the Fourth Moment Theorem. 
Our criterion gives conditions that are sufficient, not necessary, but
its simplicity makes it potentially suitable to many different contexts.

\smallskip

In Section~\ref{sec:applications} we study the partition function $Z_N^\beta$
of two-dimensional directed polymers in random environment.
In the limit $N\to\infty$, and for a suitable tuning of the 
inverse temperature $\beta = \beta_N$ (in the so-called sub-critical regime), the partition function
exhibits Edwards-Wilkinson fluctuations \cite{CSZ17b}, i.e.,
it converges to a log-correlated Gaussian field when averaged over the starting point.
An analogous result was obtained in \cite{CSZ20} for the logarithm of the partition function.
Our criterion from Section~\ref{sec:convergence}, besides
providing alternative and more elementary proofs of Edwards-Wilkinson fluctuations,
gives a natural framework to obtain new Gaussian asymptotics.
We give two main illustrations.
\begin{itemize}
\item We prove that a \emph{singular product} between the partition function and the underlying disorder
has a non-trivial Gaussian limit, see Theorem~\ref{th:singular}.
This result sheds light on the mechanism 
which produces Edwards-Wilkinson fluctuations,
explaining the source of the non-trivial factor which arises in the limiting equation.

\item For the partition function $Z_N^\beta$
with a fixed starting point, we obtain an \emph{explicit chaos expansion} 
$X_N^\dom$ which sharply approximates
$\log Z_N^\beta$, see Theorem~\ref{th:apprZ};
then we prove that \emph{$X_N^\dom$, hence $\log Z_N^\beta$ too,
is asymptotically Gaussian}, see Theorem~\ref{th:XNGauss}.
We thus recover the main result in \cite{CSZ17b} with a simpler and more conceptual proof.
\end{itemize}
These results can also be formulated in the continuum
setting of the SHE and KPZ equation. We refer to Subsection~\ref{sec:perspectives}
for a discussion and further perspectives.

\smallskip

The following Sections~\ref{sec:proof-poly}--\ref{sec:XNGauss}
contain the proofs of our main results, while some
technical lemmas have been deferred to Appendix~\ref{sec:technical}.

\section{Gaussian limits for polynomial and Wiener chaos}
\label{sec:convergence}

Our general convergence results can be phrased in a discrete setting (polynomial chaos) and in a
continuum one (Wiener chaos). We start with the former, which is more elementary.

\subsection{Polynomial chaos}

Let $\mathbb{T}$ be a countable set.
For each $N\in\N$, we consider a family $\eta^N = (\eta^N_t)_{t \in \mathbb{T}}$
of \emph{independent random variables}, not necessarily identically distributed,
with zero mean and unit variance:
\begin{equation}  \label{eq:disordergen}
  \mathbb{E} [\eta^N_t] = 0 \,, \qquad \mathbb{E} [(\eta^N_t)^2] = 1 \,.
\end{equation}
We further require the \emph{uniform integrability of the squares}:
\begin{equation}\label{eq:ui}
  \lim_{L \rightarrow \infty} \ \sup_{N \in \mathbb{N}, \; t \in \mathbb{T}} \
  \mathbb{E} \Big[| \eta^N_t |^2\,  \ind_{\{ | \eta_t^N | > L \}} \Big] = 0 \,,
\end{equation}
which follows from \eqref{eq:disordergen} if the
$\eta^N_t$'s have the same distribution. 
In general, a sufficient easy condition for \eqref{eq:ui} is that 
$\sup_{N,t} \bbE[|\eta^N_t|^p] < \infty$ for some $p>2$.

\smallskip

We consider a sequence of random variables $(X_N)_{N\in\N}$ that are polynomial chaos, i.e.\
multi-linear polynomials in the $\eta^N_t$'s. More precisely, we assume that
\begin{equation}\label{eq:polychaosgen}
  X_N = \sum_{A \subset \mathbb{T}} q_N (A)  \,\eta^N (A) \,, \qquad \text{with} \qquad
  \eta^N (A) := \prod_{t \in A} \eta^N_t \,,
\end{equation}
where $q_N(\cdot)$ are real coefficients
and the sum ranges over \emph{finite nonempty subsets $A \subset \bbT$}
(i.e.\ $q_N(A) \ne 0$ only if $0 < |A| < \infty$).
We can split the sum according to the cardinality $k$ of the subset~$A$: if we write $A= \{t_1, \ldots, t_k\}$ for distinct points $t_i \in \T$, we can rewrite \eqref{eq:polychaosgen} as
\begin{equation}\label{eq:polychaosgen-k}
  X_N = \sum_{k=1}^\infty \, 
  \sum_{\substack{\{t_1,\ldots, t_k\} \subset \mathbb{T}\\
  t_i \ne t_j \, \forall i \ne j}} q_N (\{t_1,\ldots, t_k\})  
  \,\prod_{i=1}^k \eta^N_{t_i} \,.
\end{equation}
We assume that $\sum_{A \subset\bbT} q_N(A)^2 < \infty$, so that $X_N$
is a well-defined random variable with
\begin{equation} \label{eq:moments12}
	\mathbb{E} [X_N] = 0, \qquad \mathbb{E} [X_N^2] = \sum_{A \subset \mathbb{T}} 
	q_N(A)^2  \,,
\end{equation}
because $(\eta^N(A))_{A \subset \bbT}$ are
centered and orthogonal random variables in $L^2$.

\smallskip

Our goal is to prove \emph{convergence in distribution of $X_N$ toward a Gaussian random variable}.
This can be achieved via the celebrated \emph{Fourth Moment Theorem},
formulated in our context in \cite{NPR10}
and slightly extended in \cite[Theorem~4.2]{CSZ17b};
see also the previous works \cite{NP05,dJ90,dJ87,R79} and the book \cite{NP12}.
The Fourth Moment Theorem deals with a sequence $X_N$ of polynomial chaos
in a \emph{fixed order chaos} (i.e.\ a single term~$k$ in \eqref{eq:polychaosgen-k}) 
and it requires to compute the \emph{second and fourth moments of $X_N$}.

Our first main result gives sufficient conditions for convergence
to a Gaussian limit  \emph{only based on second moment assumptions on $X_N$},
which can be directly applied to a superposition of chaos of different orders.
Let us introduce the shorthand
  \begin{equation}\label{eq:sigmaB}
  \sigma^2_N (\mathbb{B}) := \sum_{A \subset \mathbb{B}} q_N (A)^2
  \qquad \text{for} \quad \bbB \subset \bbT \,,
  \end{equation}
which gives the contribution to the second moment
of $X_N$ of the subsets of $\bbB$  (recall \eqref{eq:moments12}).
We can formulate our conditions as follows.
\begin{enumerate}
  \item\label{hyp:1disc} \emph{Limiting second moment}:
  \begin{equation} \label{hyp1}
    \lim_{N \rightarrow \infty} \sigma^2_N(\bbT) = 
    \lim_{N \rightarrow
    \infty} \ \sum_{A \subset \mathbb{T}} q_N (A)^2 \,=\, \sigma^2 \in (0, \infty) \,  ,
  \end{equation}
i.e.\ the second moment of $X_N$ converges to a finite limit.

  \item\label{hyp:2disc} \emph{Subcriticality}:
  \begin{equation}
   \lim_{K\to \infty} \ \limsup_{N \to \infty} \ \sum_{\substack{A \subset
    \mathbb{T} \\ |A| > K}} q_N (A)^2 = 0 \,,
    \label{hyp2}
  \end{equation}
i.e.\  the contribution of high order chaos to the second moment of $X_N$ is negligible.

  \item\label{hyp:3disc} \emph{Spectral localization}: 
for any $M, N \in \mathbb{N}$ we can find
\emph{$M$ disjoint subsets (``boxes'')}:
\begin{equation*}
	\mathbb{B}_1, \ldots, \mathbb{B}_M \subset \mathbb{T} \qquad \text{with} \qquad
	\bbB_i \cap \bbB_j = \emptyset \quad \text{for } i\ne j \,,
\end{equation*}
(where $\mathbb{B}_i   =\mathbb{B}_i^{(N, M)}$ 
may depend on $N,M$) such that the following
conditions hold (recall \eqref{eq:sigmaB}):
\begin{gather}\label{hyp3b}
 \lim_{M\to\infty} \ 
  	\lim_{N \rightarrow \infty}  \ \sum_{i = 1}^M \, \sigma^2_N (\mathbb{B}_i) \,=\,
  	\sigma^2 \,,\\
\label{hyp3a}
  \lim_{M \rightarrow \infty} \ \limsup_{N \rightarrow \infty} \  \Big\{ \max_{i =
  		1, \ldots, M} \sigma^2_N (\mathbb{B}_i) \Big\} \,=\, 0 \,,
  \end{gather}
i.e.\ the main contribution to the second moment of $X_N$ comes from \emph{subsets 
contained in one of the boxes $\bbB_1, \ldots, \bbB_M$},
whose individual contribution is uniformly small.
\end{enumerate}
Note that conditions \eqref{hyp:1disc}, \eqref{hyp:2disc}, \eqref{hyp:3disc} are
\emph{second moment assumptions}.
The name ``subcriticality'' for condition \eqref{hyp:2disc}
is inspired by directed polymers, that we discuss in Section~\ref{sec:applications}, and more generally by
marginally relevant disordered systems, see \cite{CSZ17a},
which undergo a phase transition at a critical point
determined precisely by the failure of condition~\eqref{hyp2}.

\smallskip

We can now state our first main result.

\begin{theorem}[Gaussian limits for polynomial chaos]\label{th:main}
Let $X_N$ be a polynomial chaos as in \eqref{eq:polychaosgen},
with coefficients $q_N (\cdot)$ satisfying the assumptions
\eqref{hyp:1disc}, \eqref{hyp:2disc}, \eqref{hyp:3disc}
(see \eqref{hyp1}--\eqref{hyp3a}), with respect to independent
random variables $\eta^N = (\eta^N_t)_{t \in \mathbb{T}}$
which satisfy \eqref{eq:disordergen} and \eqref{eq:ui}. 
Then as $N\to\infty$ we have the convergence in distribution
\begin{equation}\label{eq:main}
	X_N \, \xrightarrow[]{\ d \ } \, \mathcal{N} (0, \sigma^2) \,.
\end{equation}
\end{theorem}

The proof is given in Section~\ref{sec:proof-poly} and comes in two steps:
\begin{itemize}
\item first we approximate $X_N$ in $L^2$ by a sum
$\sum_{i=1}^M X_{N,i}$
of \emph{independent} random variables,
for a suitable $M = M_N \to \infty$;

\item then we show that the random variables
$(X_{N,i})_{1 \le i \le M_N}$ satisfy the assumption of the
\emph{Central Limit Theorem for triangular arrays},
which eventually yields \eqref{eq:main}.
\end{itemize}
We will also replace the random variables $(\eta^N_t)$ by 
a family of random variables with bounded moments
of some order $p > 2$ (e.g.\ by Gaussians)
to exploit the hypercontractivity of polynomial chaos, see \cite{MOO10}.
The justification of this replacement will be given at the end of the proof
exploiting a suitable Lindeberg principle, see \cite{MOO10,CSZ17a}.

\begin{remark}
When the polynomial chaos $X_N$
belongs to a fixed order chaos,
the conditions of the Fourth Moment Theorem
are known to be \emph{optimal}, i.e.\ 
necessary and sufficient for the asymptotic Gaussianity of~$X_N$. 
It would be interesting to investigate how far from optimality are our conditions 
\eqref{hyp1}--\eqref{hyp3a} in this setting.
A direct comparison between our conditions and the Fourth Moment Theorem
is not straightforward, due to
the freedom in the choice of the boxes $\bbB_i$ in \eqref{hyp3b}-\eqref{hyp3a}.
\end{remark}

\subsection{Wiener chaos}

Theorem~\ref{th:main} has a direct translation for Wiener chaos.
Let $(E,\cE,\mu)$ be a Polish (complete separable metric) space, endowed with
its Borel $\sigma$-field $\cE$ and with a non-atomic measure $\mu$.
Let $\cE^* = \{A \in \cE: \ \mu(A)<\infty\}$ be the class
of measurable sets with finite measure.
By \emph{Gaussian random measure} on $(E,\cE,\mu)$
we mean a centered Gaussian process $W = (W(A))_{A \in \cE^*}$
with $\bbcov[W(A),W(B)] = \mu(A \cap B)$,
defined on some probability space $(\Omega,\cA,\bbP)$.
We often use the informal notation $W(\dd x)$.
The most important example is given by \emph{white noise},
which corresponds to $E = \R^d$ 
 with $\mu =$ Lebesgue measure.

We fix a Gaussian random measure $W(\dd x)$ on $(E,\cE,\mu)$.
For every $k\in\N$ and every real function
$f \in L^2(E^k, \mu^{\otimes k})$, by \cite{Ito51,NP12} we can define the stochastic integral
\begin{equation*}
	W^{\otimes k}(f) = \int_{E^k} f(x_1,\ldots, x_k) \, W(\dd x_1) \cdots W(\dd x_k) 
\end{equation*}
which is a centered random variable in $L^2(\Omega)$ (non Gaussian as soon as $k > 1$
and $f\not\equiv 0$).
For \emph{symmetric} functions 
$f \in L^2(E^k, \mu^{\otimes k})$ and $g \in L^2(E^{k'}, \mu^{\otimes k'})$ we have
the \emph{Ito isometry}:
\begin{equation}\label{eq:Ito}
\begin{split}
	\bbE[W^{\otimes k}(f) \, W^{\otimes k'}(g)] 
	& \,=\, \ind_{\{k=k'\}} \, k! \, 
	\langle f,g \rangle_{L^2(E^k,\mu^{\otimes k})} \\
	&\,=\, \ind_{\{k=k'\}} \, k! 
	\int_{E^k} f(x_1, \ldots, x_k) \, g(x_1, \ldots, x_k) \, \mu(\dd x_1) \cdots \mu(\dd x_k) \,.
\end{split}
\end{equation}

\smallskip

In this ``continuum setting'',
in analogy with the discrete polynomial chaos \eqref{eq:polychaosgen-k},
we consider a sequence $(\tilde X_N)_{N\in\N}$ of Wiener chaos with respect to $W(\dd x)$,
that is
\begin{equation}\label{eq:Wienerchaos}
 \tilde  X_N = \sum_{k=1}^\infty \,\int_{E^k} \, \tilde q_N(x_1, \ldots, x_k)  \,
  W(\dd x_1) \cdots W(\dd x_k) \,,
\end{equation}
where $\tilde q_N$ is a \emph{symmetric} $L^2$ function defined on $\bigcup_{k=1}^\infty
(E^k, \cE^{\otimes k}, \mu^{\otimes k})$.
Then, by \eqref{eq:Ito},
\begin{equation} \label{eq:moments12Ito}
\begin{split}
	\bbE[\tilde X_N] = 0 \,, \quad \
	\bbE[\tilde X_N^2]
	& = \sum_{k=1}^\infty k! \, \|\tilde q_N\|_{L^2(E^k)}^2 
	 = \sum_{k=1}^\infty k! 
	\, \int_{E^k} \tilde q_N (x_1,\ldots, x_k)^2 \, \mu(\dd x_1) \cdots \mu(\dd x_k) \,.
\end{split}
\end{equation}

\begin{remark}
Every centered random variable in $L^2(\Omega)$, which is measurable
with respect to the $\sigma$-algebra generated by $W$,
admits an expansion like \eqref{eq:Wienerchaos}.
\end{remark}

\begin{remark}
The factor $k!$ in \eqref{eq:moments12Ito} is due to the fact that $\tilde q_N$ 
in \eqref{eq:Wienerchaos}  is a symmetric
function of the \emph{ordered} variables $x_1, \ldots, x_k$, whereas $q_N$ in \eqref{eq:polychaosgen-k}
is a function of \emph{unordered} variables (i.e.\ subsets) $\{t_1,\ldots, t_k\}$.
To formally match \eqref{eq:polychaosgen-k}-\eqref{eq:moments12} with 
\eqref{eq:Wienerchaos}-\eqref{eq:moments12Ito}, we should identify $q_N$ with
$k! \, \tilde q_N$ and $\sum_{\{t_1,\ldots, t_k\} \subset \bbT} \, \prod_{i=1}^k \eta^N_{t_i}$ with
$\frac{1}{k!} \int_{E^k} W(\dd x_1) \cdots W(\dd x_k)$.
\end{remark}

Mimicking \eqref{eq:sigmaB}, we set
  \begin{equation}\label{sigmaBcont}
 \tilde \sigma^2_N (\mathbb{B}) := \sum_{k=1}^\infty \,
   k! \, \int_{\bbB^k} \tilde q_N (x_1,\ldots, x_k)^2 \, \mu(\dd x_1) \cdots \mu(\dd x_k)
  \qquad \text{for measurable} \ \bbB \subset E \,,
  \end{equation}
which gives the contribution to the second moment
of $\tilde X_N$ of subsets in $\bbB$, see \eqref{eq:moments12Ito}.
We can now formulate our conditions in the continuum setting.
\begin{enumerate}
\renewcommand{\theenumi}{$\tilde{\arabic{enumi}}$}%
  \item\label{hyp:1cont} \emph{Limiting second moment}:
  \begin{equation} \label{hyp1cont}
  \lim_{N \rightarrow
    \infty} \  \tilde \sigma_N^2(E) \,=\,
    \lim_{N\to\infty} \ \sum_{k=1}^\infty k! \, \|\tilde q_N\|_{L^2(E^k)}^2
    \,=\, \sigma^2 \in (0, \infty) \,  ,
  \end{equation}
i.e.\ the second moment of $\tilde X_N$ converges to a finite limit.

  \item\label{hyp:2cont} \emph{Subcriticality}:
  \begin{equation}
   \lim_{K\to \infty} \ \limsup_{N \to \infty} \ \sum_{k > K} k! \, \|\tilde q_N\|_{L^2(E^k)}^2 = 0 \,,
    \label{hyp2cont}
  \end{equation}
i.e.\  the contribution of high order chaos to the second moment of $\tilde X_N$ is negligible.

  \item\label{hyp:3cont} \emph{Spectral localization}: 
for any $M, N \in \mathbb{N}$ we can find
\emph{$M$ disjoint subsets (``boxes'')}:
\begin{equation*}
	\mathbb{B}_1, \ldots, \mathbb{B}_M \subset E \qquad \text{with} \qquad
	\bbB_i \cap \bbB_j = \emptyset \quad \text{for } i\ne j 
\end{equation*}
(where $\mathbb{B}_i   =\mathbb{B}_i^{(N, M)}$
may depend on $N,M$) such that, recalling \eqref{sigmaBcont},
\begin{gather}\label{hyp3bcont}
\lim_{M\to\infty} \
  	\lim_{N \rightarrow \infty}  \ \sum_{i = 1}^M \, \tilde \sigma^2_N (\mathbb{B}_i) \,=\,
  	\sigma^2 \,,\\
\label{hyp3acont}
  \lim_{M \rightarrow \infty} \ \lim_{N \rightarrow \infty} \  \Big\{ \max_{i =
  		1, \ldots, M} \tilde \sigma^2_N (\mathbb{B}_i) \Big\} \,=\, 0 \,,
  \end{gather}
i.e.\ the main contribution to the second moment of $\tilde X_N$ comes from \emph{subsets 
contained in one of the $M$ boxes $\bbB_1, \ldots, \bbB_M$},
whose individual contribution is uniformly small.
\end{enumerate}

\smallskip

We can finally state the version of Theorem~\ref{th:main} for Wiener chaos.
We omit the proof because it follows very closely that of Theorem~\ref{th:main},
given in Section~\ref{sec:proof-poly}.

\begin{theorem}[Gaussian limits for Wiener chaos]\label{th:maincont}
Let $\tilde X_N$ be a Wiener chaos as in \eqref{eq:Wienerchaos},
with coefficients $\tilde q_N (\cdot)$ satisfying the assumptions
\eqref{hyp:1cont}, \eqref{hyp:2cont}, \eqref{hyp:3cont}
(see \eqref{hyp1cont}--\eqref{hyp3acont}),
with respect to a Gaussian random measure $W(\dd x)$ on a Polish measure space
$(E, \cE, \mu)$. Then as $N\to\infty$ we have the convergence in distribution
\begin{equation}\label{eq:maincont}
	\tilde X_N \, \xrightarrow[]{\ d \ } \, \mathcal{N} (0, \sigma^2) \,.
\end{equation}
\end{theorem}

\section{Applications to directed polymers}
\label{sec:applications}

We now present applications of our convergence results in Section~\ref{sec:convergence}
to directed polymers in random environment on $\Z^2$.

\subsection{Directed polymers and stochastic PDEs}

Let $S=(S_n)_{n\ge 0}$ be the simple symmetric random walk on $\Z^2$,
whose law we denote by $\P$. Let $\omega = (\omega(n,x))_{n\in\N, x \in \Z^2}$ be a family
of i.i.d.\ random variables, independent of $S$,
with law $\bbP$ and such that
\begin{equation}\label{eq:disorder}
	\bbE[\omega(n,x)] = 0 \,, \qquad \bbE[\omega(n,x)^2] = 1 \,, \qquad
	\lambda(\beta) := \log \bbE[\rme^{\beta \omega(n,x)}] < \infty
	\quad \forall \beta > 0 \,.
\end{equation}
Intuitively, trajectories of the random walk $S$ represent
polymer configurations, while configurations $\omega$
describe the \emph{disorder}, which plays the role of a \emph{random environment}.
Given a scale parameter $N\in\N$, a starting time-space point $(m,z) \in \{0,\ldots, N\}\times \Z^2$ and
an interaction strength $\beta > 0$, the partition function of the directed polymer model is
\begin{equation}\label{eq:ZN}
	Z_N^{\beta}(m,z) := \E \Big[\rme^{\sum_{n=m+1}^N 
	(\beta \omega(n,S_n) - \lambda(\beta))}
	\,\Big|\, S_m = z \Big] \,.
\end{equation}
Directed polymers were originally introduced as an effective interface model 
in the framework of the Ising model with impurities, but over the years
they have become an object of independent study
and a prototype of a disorder system which is amenable to detailed rigorous investigation.
We refer to the monograph by Comets \cite{C17} for a recent account.

\smallskip

A source of interest for directed polymers is their link with the multiplicative
Stochastic Heat Equation (SHE), which is the stochastic PDE formally written as follows:
\begin{equation}\label{eq:she}
	\partial_t u(t,x) = \frac{1}{2} \Delta_x u(t,x) \,+\, \beta \, \dot{W}(t,x) \, u(t,x) \,,
\end{equation}
where $\beta > 0$ tunes the interaction strength and $\dot{W}(t,x)$ denotes 
white noise on $(0,\infty) \times \R^2$.
In one space dimension $d=1$, this equation admits a rigorous integral formulation by 
the classical Ito-Walsh integration. In higher dimensions $d \ge 2$, this approach fails
due to strong irregularity of white noise
and no obvious meaning can be given to its solution $u(t,x)$.

By the Markov property of simple random walk, the diffusively rescaled partition function 
\begin{equation} \label{eq:diffZ}
	U_N(t,x) := Z_N^{\beta}(\lfloor Nt\rfloor, \lfloor \sqrt{N}x\rfloor) 
\end{equation}
solves a discretized version of \eqref{eq:she}
(with $\partial_t$ and $\frac{1}{2}\Delta_x$ replaced by
$-\partial_t$ and $\frac{1}{4} \Delta_x$,
see \eqref{eq:she-discrete2} below).
This explains the interest for the convergence as $N\to\infty$ of $U_N(t,x)$,
possibly for suitable $\beta = \beta_N$,
since it provides an approximation of the ill-defined SHE solution $u(t,x)$.

It is also very interesting to look at the \emph{logarithm of the partition function}
\begin{equation*}
	\log Z_N^{\beta}(\lfloor Nt\rfloor, \lfloor \sqrt{N}x\rfloor)
\end{equation*}
because it provides an approximation for the solution $h(t,x) = \log u(t,x)$
of the Kardar-Parisi-Zhang equation (KPZ),
which is the stochastic PDE formally given by
\begin{equation}\label{eq:kpz}
	\partial_t h(t,x) = \frac{1}{2} \Delta_x h(t,x) \,+\,
	\frac{1}{2} |\nabla_x h(t,x)|^2 \,+\, \beta \, \dot{W}(t,x) \
	\text{``}-\, \infty\,\text{''} \,,
\end{equation}
where the last term ``$-\infty$'' indicates a form of renormalization.

\begin{remark}[Edwards-Wilkinson equation]
The Stochastic Heat Equation \eqref{eq:she} is singular due to the
\emph{multiplicative noise term} $\dot{W} u$. 
The additive version of this equation,  known as the \emph{Edwards-Wilkinson equation},
is well-posed and reads as follows:
\begin{equation}\label{eq:EW}
	\partial_t v (t,x) = \frac{s}{2} \Delta_x v(t,x) + c \, \dot{W}(t,x) \,,
\end{equation}
where $s > 0$ and $c \in \R$ are given parameters.
Starting from $v(0,\cdot) \equiv 0$,
the solution $v = v^{(s,c)}$ 
is a random \emph{distribution} (i.e.\ generalized function)
which is Gaussian 
with explicit covariance,
see \cite[Remark~1.5]{CSZ20}.
More precisely, if we denote by $\langle v^{(s,c)}, \psi\rangle$
the pairing between the distribution $ v^{(s,c)}$ and a test function
$\psi$, which formally corresponds to
\begin{equation}\label{eq:average}
	\langle v^{(s,c)}, \psi\rangle := \int_{\R^2} v^{(s,c)}(t,x) \, \psi(t,x) \, \dd t \, \dd x \,,
\end{equation}
then $\langle v^{(s,c)}, \psi\rangle$ for $\psi \in C^\infty_c([0,\infty) \times \R^2)$
is a centered Gaussian process with
\begin{equation}\label{eq:covEW}
\begin{split}
	\bbcov\big[ \langle v^{(s,c)}, \psi\rangle, \,  \langle v^{(s,c)}, \psi'\rangle \big]
	= \int_{([0,\infty)\times\R^2)^2} \psi(t,x) \, K_{t,t'}^{(s,c)}(x,x') \, \psi'(t',x')
	\, \dd t \, \dd x \, \dd t' \, \dd x' \,,
\end{split}
\end{equation}
where the covariance kernel is given by
\begin{equation}\label{eq:covEW2}
	K_{t,t'}^{(s,c)}(x,x') 
	:= \frac{s \, c^2}{2} \, \int_{s|t-t'|}^{s(t+t')}
	g_{u}(x-x') \, \dd u \,, \qquad \text{where} \qquad
	g_u(y) := \frac{\rme^{-\frac{|y|^2}{2u}}}{2\pi u} \,.
\end{equation}
\end{remark}

\subsection{Edwards-Wilkinson fluctuations}
Let us define 
\begin{gather}
	\label{eq:uN}
	u_n := \sum_{z\in\Z^2} \P(S_n=z)^2 = \P(S_{2n}=0) \sim \frac{1}{\pi} \, \frac{1}{n} \,, \\
	\label{eq:RN}
	R_N := \sum_{n=1}^N \sum_{z\in\Z^2} \P(S_n=z)^2
	= \sum_{n=1}^N u_n \sim \frac{1}{\pi} \log N \,,
\end{gather}
where the asymptotic relations (respectively as $n\to\infty$ and as $N\to\infty$)
follow by the local central limit theorem (see \eqref{eq:llt} below).
Henceforth we are going to fix $\beta = \beta_N$ given by
\begin{equation} \label{eq:sub-critical}
	\beta_N  := \frac{\hat\beta}{\sqrt{R_N}} \sim \frac{\hat\beta \, \sqrt{\pi}}{\sqrt{\log N}}  \qquad
	\text{with} \qquad \hat\beta \in (0, 1) \,,
\end{equation}
also known as the \emph{sub-critical regime}. This ensures that the partition function $Z_N^{\beta_N}$
has a bounded second moment as $N\to\infty$, see \cite{CSZ17b}. 
It was recently shown in \cite{LZ21+,CZ21+} that in fact \emph{all moments of $Z_N^{\beta_N}$ 
are bounded
in this regime}.

\smallskip

We look at the fluctuations of the diffusively rescaled partition function, encoded by
\begin{equation} \label{eq:JN}
	V_N(t,x) := \frac{1}{\beta_N} \big( 
	Z_N^{\beta_N}(\lfloor N t \rfloor, \lfloor \sqrt{N}x\rfloor)
	- 1 \big) \qquad \text{for} \quad (t,x) \in [0,1] \times \R^2 \,.
\end{equation}
It was shown in \cite[Theorem~2.13]{CSZ17b} that 
$Z_N^{\beta_N}$ exhibits \emph{Edwards-Wilkinson fluctuations}, because
$V_N(t,x)$ converges as $N\to\infty$
to a solution of the Edwards-Wilkinson equation~\eqref{eq:EW}:
\begin{equation}\label{eq:cdef}
	V_N(t,x) \, \overset{\cD}{\Longrightarrow} \,
	\tilde v(t,x) := v^{(\frac{1}{2}, c_{\hat\beta})}(1-t , x) \qquad \text{where} \quad
	c_{\hat\beta} := \sqrt{\frac{1}{1-\hat\beta^2}}  \,,
\end{equation}
where ``$\overset{\cD}{\Longrightarrow}$'' denotes convergence in law
\emph{as a random distribution}:\footnote{By the Cram\'er-Wold device
\cite[Theorem~29.4]{Bil95},
relation \eqref{eq:EW-basic} implies convergence of all finite-dimensional distributions
of the random field $(\langle V_N, \psi \rangle)_\psi$
toward $\langle \tilde v, \psi \rangle$.}
for $\psi \in C_c([0,1]\times \R^2)$
\begin{equation} \label{eq:EW-basic}
\begin{split}
	\langle V_N, \psi \rangle
	:= \int_{\R \times\R^2} V_N(t,x) \, \psi(t,x) \, \dd t \, \dd x 
	\ \xrightarrow[]{\ d \ } \ 
	\langle \tilde v, \psi \rangle  \,.
\end{split}
\end{equation}
The convergence \eqref{eq:cdef} was  proved
in \cite{CSZ17b} using the Fourth Moment Theorem,
based on a polynomial chaos expansion of the partition function,
see \eqref{eq:Zpoly} below.
Remarkably, our Theorem~\ref{th:main} allows for an
\emph{alternative and more elementary proof of \eqref{eq:cdef},
based on second moments calculations}.
The details will be presented in~\cite{cf:Cottini}.

\begin{remark}\label{rem:var12}
The factor $\frac{1}{2}$ 
in the parameters of $\tilde v(t,x) = v^{(\frac{1}{2}, c_{\hat\beta})}(1-t,x) $,
see \eqref{eq:cdef}, is due to the fact that 
$\E[S_1^{(i)}, S_1^{(j)}] = \frac{1}{2} \ind_{i=j}$ for $i,j \in \{1,2\}$.
In view of \eqref{eq:EW}, note that $\tilde v$ satisfies
\begin{equation} \label{eq:EWmod}
	-\partial_t \tilde v(t,x) = \frac{1}{4} \Delta_x \tilde v(t,x) + c_{\hat\beta} \,
	\dot{W}(t,x) \,.
\end{equation}
\end{remark}

Edwards-Wilkinson fluctuations also hold for the logarithm of the partition function,
suitably centered and rescaled as in \eqref{eq:JN}:
\begin{equation} \label{eq:HN}
	H_N(t,x) := \frac{1}{\beta_N} \Big( 
	\log Z_N^{\beta_N}(\lfloor Nt\rfloor, \lfloor \sqrt{N}x\rfloor)
	- \bbE\big[ \log Z_N^{\beta_N}(\lfloor Nt\rfloor, \lfloor \sqrt{N}x\rfloor) \big] \Big) \,.
\end{equation}
Indeed, it was shown in \cite[Theorem~1.6]{CSZ20} that
a precise analogue of \eqref{eq:cdef} holds:
\begin{equation}\label{eq:cdef2}
	H_N(t,x) \, \overset{\cD}{\Longrightarrow} \,
	\tilde v(t,x) = v^{(\frac{1}{2}, c_{\hat\beta})}(1-t , x) \,.
\end{equation}
This convergence was in fact \emph{deduced}
in \cite{CSZ20} from \eqref{eq:cdef} by
means of a highly non trivial linearization procedure.
The alternative and more elementary proof of \eqref{eq:cdef}
based on our Theorem~\ref{th:main} can then be transferred to yield
a proof of \eqref{eq:cdef2} as well.
We refrain from giving the details, which will be presented
in~\cite{cf:Cottini}. 

\begin{remark}\label{rem:literature}
A simultaneous and independent proof of \eqref{eq:cdef2}
was given in~\cite{G20} for small $\hat\beta > 0$ in a closely related
context, namely for the KPZ equation \eqref{eq:kpz}
where the noise $\dot{W}(t,x)$ is regularized by mollification (rather than by
discretization, as we consider here). Previously,
the existence of non-trivial subsequential limits had been shown in~\cite{CD20}.
We refer to \cite{DG20+,NN21+} for some recent extensions and generalizations.
\end{remark}

In this paper, we exploit Theorem~\ref{th:main}
to prove two new Gaussian convergence results related to the partition function,
that we now describe.

\smallskip

\subsection{Main result I (singular product)}

The diffusively rescaled partition function $U_N(t,x)$ in \eqref{eq:diffZ}
approximates the solution of the Stochastic Heat Equation
\eqref{eq:she} with \emph{multiplicative} noise.
It is not clear a priori why the fluctuations of $U_N(t,x)$, encoded by $V_N(t,x)$
in \eqref{eq:JN}, converge to $\tilde v(t,x)$ which solves the Stochastic Heat Equation
with  \emph{additive} noise,
see \eqref{eq:EWmod}, with an intensity $c_{\hat\beta}$
which \emph{explodes} as $\hat\beta \uparrow 1$.
We now present a result which sheds light on the mechanism which leads to \eqref{eq:EWmod}.

\smallskip

Let us introduce a modified 
disorder $\eta_N = (\eta_N(m,z))_{m\in\N, z\in\Z^2}$, recalling \eqref{eq:disorder}:
\begin{equation} \label{eq:eta}
	\eta_N(m,z) := \frac{\rme^{\beta_N \omega(m,z) - \lambda(\beta_N)}-1}{\sigma_N}
	\qquad \text{where} \qquad 
	\sigma_N^2 := \rme^{\lambda(2\beta_N)-2\lambda(\beta_N)}-1 
	\underset{N\to\infty}{\sim} \beta_N^2 \,.
\end{equation}
We denote by $\dot{W}_N(t,x)$, for $t > 0$, $x \in \R^2$,
the diffusively rescaled version of $\eta_N$:
\begin{equation} \label{eq:apprwhite}
	\dot{W}_N(t,x) := N \, \eta_N(\lfloor Nt\rfloor, \lfloor \sqrt{N}x\rfloor)  \,.
\end{equation}
For any $N\in\N$, the modified disorder
$\eta_N = (\eta_N(m,z))_{m\in\N, z\in\Z^2}$ is i.i.d.\ with
$\bbE[\eta_N(m,z)]=0$ and
$\bbE[\eta_N(m,z)^2] 
= 1$, see \eqref{eq:disorder}, and higher moments of $\eta_N$
are uniformly bounded (see \cite[eq. (6.7)]{CSZ17a}).
It follows that $\dot{W}_N$ converges in law to the white noise:
\begin{equation}\label{eq:convwhite}
	\dot{W}_N(t,x) \, \overset{\cD}{\Longrightarrow} \,
	\dot{W}(t,x) \,,
\end{equation}
that is $\langle \dot{W}_N, \psi \rangle \overset{d}{\to} \langle \dot{W}, \psi \rangle
\sim \cN(0, \|\psi\|_{L^2}^2)$  as $N\to\infty$,
for $\psi \in C^\infty_c([0,1]\times \R^2)$.

We now consider the product between $\dot{W}_N$ and $U_N(t,x)-1$, i.e.\ the
centered and diffusively rescaled partition function
$Z_N^{\beta_N}(\lfloor Nt\rfloor, \lfloor \sqrt{N}x\rfloor) - 1$,
see \eqref{eq:diffZ}:
\begin{equation} \label{eq:LN}
\begin{split}
	\Xi_N(t,x) 
	& := \dot{W}_N(t,x)  \, \big( U_N(t,x) - 1 \big)\\
	& \,= \beta_N  \,  \dot{W}_N(t,x) \, V_N(t,x)\,,
\end{split}
\end{equation}
where we recall that $V_N(t,x) = \beta_N^{-1}(U_N(t,x) - 1)$ is defined in \eqref{eq:JN}.

We know that $V_N \overset{\cD}{\Longrightarrow} \tilde v$ 
and $\dot{W}_N \overset{\cD}{\Longrightarrow} W$ as $N\to\infty$,
see~\eqref{eq:EW-basic} and~\eqref{eq:convwhite}. Since $\beta_N \to 0$, 
one could expect that $\Xi_N \overset{\cD}{\Longrightarrow} 0$, but \emph{this turns out to be false}.
The point is that $V_N$ and $\dot{W}_N$ only converge as random distributions,
and the product of distributions is not a continuous operation
(it is generally not even defined). The following result shows
that $\Xi_N$ has in fact a non-trivial limit as $N\to\infty$.
We prove it in Section~\ref{sec:singular}
as an application of our Theorem~\ref{th:main}.

\begin{theorem}[White noise from singular product]\label{th:singular}
Let $\beta = \beta_N$ be fixed as in \eqref{eq:sub-critical},
and set $c_{\hat\beta} := (1-\hat\beta^2)^{-1/2}$.
As $N\to\infty$, we have the joint convergence in law:
\begin{equation*}
	(\dot{W}_N, \Xi_N) \,\overset{\cD}{\Longrightarrow}\,
	\Big(\dot{W}, \, \sqrt{c_{\hat\beta}^2-1} \,\dot{W}' \Big) \,,
\end{equation*}
where $\dot{W}$ and $\dot{W}'$ denote two independent white noises on $[0,1]\times \R^2$.
More precisely, for any $\psi \in C^\infty_c([0,1]\times\R^2)$,
the following joint convergence in distribution holds:
\begin{equation*}
	\big( \langle\dot{W}_N,\psi\rangle ,\,
	\langle \Xi_N, \psi\rangle \big) \,\xrightarrow[]{\ d \ }\,
	\cN\big(0, \|\psi\|_{L^2}^2 \, \Sigma_{\hat\beta} \big) \qquad \text{where} \qquad
	\Sigma_{\hat\beta} = \left(\begin{matrix}
	1 & 0 \\
	0 &  c_{\hat\beta}^2-1 
	\end{matrix}\right) \,.
\end{equation*}
\end{theorem}

We can finally give a heuristic explanation for equation \eqref{eq:EWmod}. 
One can check that $ Z_N^{\beta_N}(m,z)$ in \eqref{eq:ZN} solves the following
\emph{difference equation}, for $m \le N$ and $z\in\Z^2$:
\begin{equation} \label{eq:she-discrete}
	 Z_N^{\beta_N}(m-1,z) -  Z_N^{\beta_N}(m,z) = \frac{1}{4} \Delta_{\Z^2}
	 Z_N^{\beta_N}(m,z) \,+\, \sigma_N \, \frac{1}{4}
	 \sum_{z' \sim z} \eta_N(m,z')  \, Z_N^{\beta_N}(m,z') \,,
\end{equation}
where  $z' \sim z$  means $z' \in \{z \pm (1,0), z \pm (0,1)\}$ and
$\Delta_{\Z^2}f(z) := \sum_{z' \sim z} \{f(z')-f(z)\}$
denotes the lattice Laplacian (we recall that
$\sigma_N$ and $\eta_N(m,z)$ are defined in \eqref{eq:eta}).

By \eqref{eq:JN} and \eqref{eq:apprwhite},
we can rewrite \eqref{eq:she-discrete} as follows,
for $(t,x) \in ((0,1] \cap \frac{\Z}{N}) \times 
(\R^2 \cap \frac{\Z^2}{\sqrt{N}})$:
\begin{equation}\label{eq:she-discrete2}
	- \partial_t^{(N)} U_N(t,x) = \frac{1}{4} \Delta_x^{(N)} U_N(t, x)
	\,+\, \sigma_N \,
	\frac{1}{4} \sum_{x' \overset{N}{\sim}\, x} \dot{W}_N(t,x') \, U_N(t,x') \, ,
\end{equation}
where $x' \overset{N}{\sim} x$ means $x' \in \{x\pm (\frac{1}{\sqrt{N}},0), x\pm (0,\frac{1}{\sqrt{N}})\}$
and we define the rescaled operators 
\begin{equation*}
\begin{split}
	\partial_t^{(N)} f(t,x) &:= 
	N \big\{ f(t,x)-f(t-\tfrac{1}{N},x) \big\} \,, \\
	\Delta_x^{(N)} f(t,x) &:=
	N \sum_{x' \overset{N}{\sim}\, x}
	\big\{ f(t,x')-f(t,x) \big\} \,.
\end{split}
\end{equation*}
Note that \eqref{eq:she-discrete2} is a discretization of
the (time reversed) Stochastic Heat Equation \eqref{eq:she},
with the factor $\frac{1}{4}$ instead of $\frac{1}{2}$
(see Remark~\ref{rem:var12}) and with $\sigma_N \sim \beta_N$ in place of~$\beta$.

We now consider $V_N(t,x) = \beta_N^{-1}(U_N(t,x)-1)$, see \eqref{eq:cdef}.
By \eqref{eq:she-discrete2} we obtain
\begin{equation}\label{eq:she-discrete-fluct}
	- \partial_t^{(N)}  V_N(t,x) = \frac{1}{4} \Delta_x^{(N)}  V_N(t, x)
	+ \frac{\sigma_N}{\beta_N} \, \frac{1}{4} \sum_{x' \overset{N}{\sim}\, x} 
	\bigg\{ \dot{W}_N(t,x')
	\,+\, \beta_N \, \dot{W}_N(t,x') \, V_N(t,x') \bigg\} \, .
\end{equation}
The last term $\beta_N \, \dot{W}_N(t,x') \, V_N(t,x')$
is nothing but $\Xi_N(t,x')$ in \eqref{eq:LN}, which
formally vanishes as $N\to\infty$ but
actually \emph{converges to an independent white noise
$\sqrt{c_{\hat\beta}^2-1} \, \dot{W}'(t,x)$}, by Theorem~\ref{th:singular}
(note that $x' \overset{N}{\sim} x$ implies $|x'-x| = 1/\sqrt{N} \to 0$).
If we assume that $V_N(t,x)$ converges to a limit $\tilde v(t,x)$,
by taking the formal limit
of \eqref{eq:she-discrete-fluct} we finally obtain
\begin{equation}\label{eq:EWtrue}
	- \partial_t \tilde v(t,x) = \frac{1}{4} \Delta_x \tilde v(t,x) \,+\,
	\dot{W}(t,x) \,+\, \sqrt{c_{\hat\beta}^2-1} \, \dot{W}'(t,x) \,.
\end{equation}
Note that \emph{this is equivalent to \eqref{eq:EWmod}}, because 
$\dot{W}(t,x) \,+\, \sqrt{c_{\hat\beta}^2-1} \, \dot{W}'(t,x)
\,\overset{d}{=}\, c_{\hat\beta} \, \dot{W}(t,x)$.

\smallskip

In conclusion, Theorem~\ref{th:singular} provides an intuitive explanation why the random field
$\tilde v(t,x)$ to which $V_N(t,x)$ converges should satisfy the equation
\eqref{eq:EWmod}, or more precisely \eqref{eq:EWtrue}. The factor $c_{\hat\beta}$ in \eqref{eq:EWmod} 
arises from the
\emph{singular product} $\Xi_N(t,x) = \beta_N \, \dot{W}_N(t,x) \, V_N(t,x)$
which gives rise to an \emph{independent white noise}, by Theorem~\ref{th:singular}. 

This result is the first step toward
a \emph{``robust analysis''} of the two-dimensional SHE \eqref{eq:she},
which would allow for a rigorous derivation of \eqref{eq:EWtrue} 
from~\eqref{eq:she-discrete-fluct}.

\subsection{Main result II (log-normality)}

So far we have discussed the distribution of
the partition function $Z_N^{\beta_N}(m,z)$, suitably rescaled,
as a \emph{random field}, i.e.\ averaging over the starting point $(m,z)$.
We now look at the distribution
of $Z_N^{\beta_N}(m,z)$ for a \emph{fixed} starting point: we fix $(m,z) = (0,0)$
by stationarity and we set
\begin{equation} \label{eq:ZNbetaN}
	Z_N^{\beta_N} := Z_N^{\beta_N}(0,0) \,.
\end{equation}
It was shown
in \cite[Theorem~2.8]{CSZ17b} that $Z_N^{\beta_N}$ is \emph{asymptotically log-normal}:
\begin{equation}\label{eq:log-norm}
\begin{gathered}
	\log Z_N^{\beta_N}
	\,\xrightarrow[]{\ d \ }\, \cN\big(-\tfrac{1}{2} \sigma_{\hat\beta}^2, \, 
	\sigma_{\hat\beta}^2 \big) \qquad
	\text{where} \qquad
	\sigma_{\hat\beta}^2 = \log c_{\hat\beta}^2
	= \log \tfrac{1}{1-\hat\beta^2}  \,.
\end{gathered}
\end{equation}

The original proof of this result, based on the Fourth Moment Theorem,
is long and technical. Our goal is to provide a less technical and more insightful proof,
based on second moment computation, exploiting our Theorem~\ref{th:main}.
The problem is that, unlike for $Z_N^{\beta_N}$, \emph{we do not have a 
polynomial chaos expansion for $\log Z_N^{\beta_N}$},
which is essential for Theorem~\ref{th:main}.
We solve this problem by first proving a result of independent
interest, which shows that $\log Z_N^{\beta_N}$ is sharply approximated in $L^2$
by an explicit polynomial chaos expansion $X_N^{\dom}$.

\smallskip

We need some setup.
We recall that the modified disorder $(\eta_N(n,x))_{n\in\N, x\in\Z^2}$ 
was defined in \eqref{eq:eta}.
We also introduce the transition kernel of the simple random walk:
\begin{equation} \label{eq:qn}
	q_n(x) := \P(S_n = x \,|\, S_0=0) 
\end{equation}
and we recall the polynomial chaos expansion of the partition function \cite{CSZ17a}:
\begin{equation}\label{eq:Zpoly}
\begin{split}
	Z_N^{\beta_N}(m,z) := 1 + \sum_{k=1}^\infty (\sigma_N)^k
	& \sum_{\substack{m = n_0 < n_1 < \ldots < n_k \le N \\
	x_0 := z, \, x_1, \ldots, x_k \in \Z^2}} \
	\prod_{i=1}^k q_{n_i-n_{i-1}}(x_i-x_{i-1}) \, \eta_N(n_i, x_i) \,.
\end{split}
\end{equation}
We define a new polynomial chaos expansion $X_N^{\dom}$,
obtained from the centered partition function $Z_N^{\beta_N} -1 = Z_N^{\beta_N}(0,0)-1$ 
imposing the constraint that \emph{all increments $n_i - n_{i-1}$
for $i \ge 2$ are dominated by the first time $n_1$}:
\begin{equation}\label{eq:XNdom}
\begin{split}
	X_N^{\dom} := \sum_{k=1}^\infty (\sigma_N)^k \!\!\!\!\!\!
	& \sum_{\substack{0 = n_0 < n_1 < \ldots < n_k \le N: \\
	\max\{n_2-n_1, n_3-n_2, \ldots, n_k - n_{k-1}\} \le n_1 \\
	x_0 := 0, \, x_1, \ldots, x_k \in \Z^2}}  \!\!
	\prod_{i=1}^k q_{n_i-n_{i-1}}(x_i-x_{i-1}) \, \eta_N(n_i, x_i) \,.
\end{split}
\end{equation}
Our key approximation result 
shows that \emph{$X_N^\dom$ is a sharp approximation
of $\log Z_N^{\beta_N}$}. The reason why this approximation is possible will be clear
in the proof, but one can already give a look at equation~\eqref{Z2multform},
which shows that a natural approximation of $Z_N^{\beta_N}$ has
a \emph{product structure}, where (a restricted version of) $X_N^{\dom}$ appears.

\begin{theorem}[Polynomial chaos for $\log Z$]\label{th:apprZ}
Set $\beta = \beta_N$ as in \eqref{eq:sub-critical}. Then 
\begin{equation} \label{eq:apprZ}
	\lim_{N\to\infty} \ \big\| \log Z_N^{\beta_N} - \big\{ X_N^{\dom} -
	\tfrac{1}{2} \, \bbE [(X_N^{\dom})^2] \big\} \big\|_{L^2} \,=\, 0 \,.
\end{equation}
\end{theorem}

\noindent
We then show, by our general Theorem~\ref{th:main},
that $X_N^\dom$ is asymptotically Gaussian.

\begin{theorem}[Asymptotic Gaussianity of $X_N^\dom$]\label{th:XNGauss}
Set $\beta = \beta_N$ as in \eqref{eq:sub-critical}. Then 
\begin{equation} \label{eq:XNGauss}
	\lim_{N\to\infty} \bbE \big[(X_N^{\dom})^2\big] = \sigma_{\hat\beta}^2
	= \log \tfrac{1}{1-\hat\beta^2}  
	\qquad \text{and} \qquad
	X_N^{\dom} \xrightarrow[]{\ d \ } \cN\big( 0, \sigma_{\hat\beta}^2 \big)\,.
\end{equation}
\end{theorem}

We prove Theorems~\ref{th:apprZ} and~\ref{th:XNGauss}
in Sections~\ref{sec:apprZ} and~\ref{sec:XNGauss}. Note that
relations \eqref{eq:apprZ} and \eqref{eq:XNGauss} together provide
a strengthening of the asymptotic log-normality of $ Z_N^{\beta_N}$, see
\eqref{eq:log-norm}.

\subsection{Conclusions and perspectives}
\label{sec:perspectives}
We discussed several convergences to a Gaussian limit
for directed polymers: the Edwards-Wilkinson fluctuations
\eqref{eq:cdef} and \eqref{eq:cdef2}, the singular product in Theorem~\ref{th:singular}
and the asymptotic log-normality in Theorem~\ref{th:XNGauss}.
We stress that these results hold in the \emph{sub-critical regime} \eqref{eq:sub-critical}
with $\hat\beta < \hat\beta_c = 1$, while they break down
in the critical regime $\hat\beta = 1$
(note that $c_{\hat\beta} \to \infty$ and $\sigma_{\hat\beta} \to \infty$ as $\hat\beta \uparrow 1$).

It would be  interesting to understand whether these results
can be suitably extended to a ``nearly critical regime'', i.e.\
when one takes $\hat\beta = \hat\beta_N \uparrow 1$ slowly enough, strictly
below the \emph{critical window} $\hat\beta = 1 + O(\frac{1}{\log N})$ studied in
\cite{BC98,GQT21,CSZ19b,CSZ21+}. We plan to investigate this issue
in future work, building on the new proofs that we presented
in this paper, which are more robust and suitable for generalization.

Another direction of research is about higher dimensions $d \ge 3$.
The Edwards-Wilkinson fluctuations \eqref{eq:cdef} and \eqref{eq:cdef2}
have been proved for $d \ge 3$ in the so-called ``$L^2$ regime'' in \cite{LZ20+}
and \cite{CNN20+}, sharpening previous work from \cite{MU18,GRZ18,CCM20,DGRZ20};
see also \cite{CCM21+} for related recent results.
It would be interesting to apply the approach of our paper in this higher dimensional
context, to check whether it is possible to go slightly beyond the ``$L^2$ regime''
(cf.\ the ``nearly critical regime'' mentioned above for~$d=2$).

Finally, we point out that many of the cited works focus on the ``continuum setting''
of the SHE \eqref{eq:she} and KPZ equation \eqref{eq:kpz} where 
the noise $\dot{W}(t,x)$ is mollified (see also Remark~\ref{rem:literature}). 
Our results of this section are formulated in the discrete setting of directed polymers,
which correspond to the stochastic PDEs \eqref{eq:she} and \eqref{eq:kpz}
where the noise $\dot{W}(t,x)$ is \emph{discretized} rather than \emph{mollified}, but
we stress that our approach can
also be applied to the continuum setting with mollification, using Theorem~\ref{th:maincont}
instead of Theorem~\ref{th:main}.

\section{Proofs of Theorem~\ref{th:main}}
\label{sec:proof-poly}

As a preliminary step to prove  Theorem~\ref{th:main},
\emph{we replace the random variables $(\eta^N_t)_{t\in\bbT}$ 
in the definition \eqref{eq:polychaosgen} of $X_N$
by independent standard Gaussians}.
We will show in Subsection~\ref{sec:Lindeberg-principle} that
such a replacement does not affect the asymptotic distribution of $X_N$ as $N\to\infty$.

We therefore assume that $\eta^N_t \sim \cN(0,1)$.
We then exploit the \emph{hypercontractivity of polynomial chaos},
which allows us to bound moments of order $p>2$ in terms of second moments, 
see \cite[Section~3.2]{MOO10} and \cite[Theorem~5.1]{J97}:
\begin{equation}\label{eq:hypercontractivity}
	\forall p > 2: \qquad
	\bbE\bigg[ \bigg| \sum_{A\subset\bbT} q_N(A) \, \eta^N(A) \bigg|^p \bigg]
	\le \bigg( \sum_{A\subset\bbT} (p-1)^{|A|} \, q_N(A)^2 \bigg)^{\frac{p}{2}} \,.
\end{equation}

\begin{remark}\label{rem:hypercontractivity}
The choice of a Gaussian distribution for the $\eta^N_t$'s 
is not fundamental here: hypercontractivity of polynomial chaos 
holds for \emph{arbitrary distributions of the $\eta^N_t$'s with
uniformly bounded moments}:
if $\sup_{N, t} \bbE[|\eta^N_t|^{\bar p}] < \infty$
for some $\overline{p}>p$, then
\begin{equation}\label{eq:hypercontractivity2}
	\bbE\bigg[ \bigg| \sum_{A\subset\bbT} q_N(A) \, \eta^N(A) \bigg|^p \bigg]
	\le \bigg( \sum_{A\subset\bbT} C_p^{|A|} \, q_N(A)^2 \bigg)^{\frac{p}{2}} \,,
\end{equation}
for a suitable $C_p < \infty$ with $\lim_{p\downarrow 2} C_p = 1$: see \cite[Theorem~B.1]{CSZ20}.
\end{remark}

\subsection{Preparation}

We consider a sequence of polynomial chaos $X_N$,
with coefficients $q_N(\cdot)$ as in \eqref{eq:polychaosgen}, which
satisfy assumptions \eqref{hyp:1disc}, \eqref{hyp:2disc}, \eqref{hyp:3disc},
see the equations \eqref{hyp1}-\eqref{hyp3a}. We now build
two suitable diverging sequences of integers $M_N \to \infty$, $K_N \to \infty$.
\begin{itemize}
\item We fix $M_N \to \infty$ slowly enough so that assumption \eqref{hyp:3disc} still
holds with $M = M_N$.
More explicitly, for every $N\in\N$ we can find disjoint subsets (``boxes'') $\bbB_i = \bbB_i^{(N)}$:
\begin{equation*}
	\mathbb{B}_1, \ldots, \mathbb{B}_{M_N} \subset \mathbb{T} \qquad \text{with} \qquad
	\bbB_i \cap \bbB_j = \emptyset \quad \text{for } i\ne j \,,
\end{equation*}
such that the following versions of \eqref{hyp3b}-\eqref{hyp3a} hold:
\begin{equation}\label{eq:hyp3+}
  	\lim_{N \rightarrow \infty}  \ \sum_{i = 1}^{M_N} \, \sigma^2_N (\mathbb{B}_i) \,=\,
  	\sigma^2 \qquad \text{and} \qquad
	\lim_{N \rightarrow \infty} \  \Big\{ \max_{i = 1, \ldots, M_N}  \,
	\sigma^2_N (\mathbb{B}_i) \Big\} \,=\, 0 \,.
\end{equation}

\item By the second relation in \eqref{eq:hyp3+},
we can fix $K_N \to \infty$ slowly enough so that
\begin{equation}\label{eq:KN}
	\lim_{N\to\infty} \ 8^{K_N} \, \max_{i = 1, \ldots, M_N}  \,
	\sigma^2_N (\mathbb{B}_i)  \,=\, 0 \,.
\end{equation}
The reason for this specific choice will be clear later, see the discussion
after \eqref{eq:seethedisc}.
Note that by our assumption \eqref{hyp:2disc}, see \eqref{hyp2}, 
for any $K_N \to \infty$ we have
\begin{equation}\label{eq:hyp2+}
	\lim_{N \to \infty} \ \sum_{\substack{A \subset
    \mathbb{T} \\ |A| > K_N}} q_N (A)^2 = 0 \,.
\end{equation}

\end{itemize}

\begin{remark}\label{M_N}
It is standard 
to deduce \eqref{eq:hyp3+} from \eqref{hyp3b}-\eqref{hyp3a}.
Indeed, given any real sequence $a_{N,M}$ which admits the limits
\begin{equation*}
	\lim_{M \rightarrow \infty} \ \limsup_{N \to \infty} \ a_{N,M} 
	\,=\, \lim_{M \rightarrow \infty} \ \liminf_{N \to \infty} \ a_{N,M} \,=\, \alpha \,,
\end{equation*}
we can always choose $M = M_N\to\infty$ slowly enough so that 
$\lim_{N\to\infty} a_{N,M_N} = \alpha$, as one can check directly.
Then, to obtain \eqref{eq:hyp3+} from \eqref{hyp3b}-\eqref{hyp3a}, it suffices to consider
\begin{equation*}
	a_{N,M} = \sum_{i=1}^M \sigma_N^2 \big( \bbB_i^{(N,M)} \big) \,, \qquad
	\text{resp.} \qquad
	a_{N,M} = \max_{i = 1, \ldots, M} \sigma_N^2\big( \bbB_i^{(N,M)} \big) \,.
\end{equation*}
\end{remark}

We next proceed with the actual proof of Theorem~\ref{th:main}.
We follow the two steps outlined after the statement of Theorem~\ref{th:main}:
\begin{itemize}
\item first we approximate the polynomial chaos $X_N$ in \eqref{eq:polychaosgen}
by a sum of suitable independent random variables, see Subsection~\ref{sec:approx};
\item then we apply the 
Feller-Lindeberg CLT to obtain the asymptotic Gaussianity \eqref{eq:main},
see Subsection~\ref{sec:CLT}.
\end{itemize}

\subsection{Approximation of $X_N$}
\label{sec:approx}
We recall the notation $\eta^N(A) := \prod_{t\in A} \eta^N_t$, see \eqref{eq:polychaosgen}.
We define a triangular array of random variables $(X_{N,i})_{i=1,\ldots, M_N}$ by setting
\begin{equation}\label{eq:XNi}
	X_{N,i} :=  \sum_{\substack{A \subset \mathbb{B}_i \\ 
	|A| \le K_N}} q_N(A) \, \eta^N(A) \qquad \text{for } i=1,\ldots,M_N \,,
\end{equation}
where we recall that $M_N\to\infty$ and $K_N \to \infty$ have been fixed
so that \eqref{eq:hyp3+}-\eqref{eq:hyp2+} hold.

We now show that the sum $\sum_{i=1}^{M_N} X_{N,i}$
is a good approximation of $X_N$.

\begin{lemma}\label{lemmatroncamento}
The following holds:
\begin{equation}\label{lemmanotronc}
	\lim_{N \rightarrow \infty} \ \Bigg\| \, X_{N} 
	\,-\, \sum_{i=1}^{M_N} X_{N,i} \, \Bigg\|_{L^2} \,=\, 0 \,.
\end{equation}
\end{lemma}

\begin{proof}
Let us define a modification of the random variables $X_{N,i}$ in \eqref{eq:XNi},
where we simply remove the constraint $|A| \le K_N$:
\begin{equation*}
	\tilde X_{N,i} :=  \sum_{A \subset \mathbb{B}_i} q_N(A) \, \eta^N(A) 
	\qquad \text{for } i=1,\ldots,M_N \,.
\end{equation*}
We are going to show that
\begin{equation}\label{eq:toshow}
	\lim_{N \rightarrow \infty} \ \Bigg\| \, X_{N} 
	\,-\, \sum_{i=1}^{M_N} \tilde X_{N,i} \, \Bigg\|_{L^2} \,=\, 0 \qquad \text{and} \qquad
	\lim_{N \rightarrow \infty} \ \Bigg\| \, \sum_{i=1}^{M_N} \tilde X_{N,i}
	\,-\, \sum_{i=1}^{M_N} X_{N,i} \, \Bigg\|_{L^2} \,=\, 0 \,.
\end{equation}

The first relation
is a direct consequence of our assumptions \eqref{hyp:1disc} and~\eqref{hyp:3disc}. 
Indeed, since the boxes $\bbB_i$ are disjoint, the random variable
$\sum_{i=1}^{M_N} \tilde X_{N,i}$ is the polynomial chaos where we only sum over subsets
$A \subset \bigcup_{i=1}^{M_N} \mathbb{B}_i$, hence the difference 
$X_{N} - \sum_{i=1}^{M_N} \tilde X_{N,i}$ is orthogonal in $L^2$ to
$\sum_{i=1}^{M_N} \tilde X_{N,i}$. As a consequence,
recalling also \eqref{eq:sigmaB}, we can write
\begin{equation*}
	\Bigg\| X_{N} - \sum_{i=1}^{M_N} \tilde X_{N,i} \Bigg\|_{L^2}^2
	\,=\, \big\| X_N \big\|_{L^2}^2 \,-\,
	\Bigg\| \sum_{i=1}^{M_N} \tilde X_{N,i} \Bigg\|_{L^2}^2
	\,=\, \sum_{A\subset \bbT} q_N(A)^2 \,-\,
	\sum_{i=1}^{M_N} \sigma_N^2(\bbB_i) \,,
\end{equation*}
hence the first relation in \eqref{eq:toshow} follows by
\eqref{hyp1} and the first relation in \eqref{eq:hyp3+}.

The second relation in \eqref{eq:toshow} follows by our assumption \eqref{hyp:2disc},
see \eqref{eq:hyp2+}, because
\begin{equation*}
	\Bigg\| \, \sum_{i=1}^{M_N} \tilde X_{N,i}
	\,-\, \sum_{i=1}^{M_N} X_{N,i} \, \Bigg\|_{L^2}^2
	\,=\, \sum_{i=1}^{M_N} \ \sum_{\substack{A\subset \bbB_i \\
	|A| > K_N}} q_N(A)^2 \,\le\,
	\sum_{\substack{A\subset\bbT\\ |A| > K_N}} q_N(A)^2  \,.
\end{equation*}
This completes the proof.
\end{proof}

\subsection{Asymptotic Gaussianity of $X_N$}
\label{sec:CLT}
In view of Lemma~\ref{lemmatroncamento}, to prove \eqref{eq:main}
it remains to prove the convergence in distribution
\begin{equation}\label{eq:main+}
	\sum_{i=1}^{M_N} X_{N,i} \, \xrightarrow[N \rightarrow \infty]{d} \, \mathcal{N} (0, \sigma^2) \,.
\end{equation}

Note that $(X_{N,i})_{i=1,\ldots, M_N}$
are \emph{independent} random variables with
zero mean and finite variance, see \eqref{eq:XNi},  because the boxes $\bbB_i \subset \bbT$
are disjoint. By the \emph{Central Limit Theorem for triangular arrays} 
\cite[Theorem~27.2]{Bil95}, it suffices to check
the convergence of the variance:
\begin{equation}\label{eq:varianceconv}
	\lim_{N\to\infty} \bbE \Bigg[ \Bigg( \sum_{i=1}^{M_N} X_{N,i} \Bigg)^2 \Bigg] = \sigma^2 \,,
\end{equation}
and the \emph{Lindeberg condition}:
\begin{equation}\label{eq:Lindeberg-condition}
	\forall \epsilon > 0: \qquad
	\lim_{N\to \infty} \ \sum_{i=1}^{M_N} \,
	\bbE \Big[ \big(X_{N,i}\big)^2 \, \ind_{\{|X_{N,i}| > \epsilon\}} \Big] \,=\, 0 \,.
\end{equation}

Relation \eqref{eq:varianceconv} follows by
Lemma~\ref{lemmatroncamento}, see \eqref{lemmanotronc}, and our assumption~\eqref{hyp:1disc},
see~\eqref{hyp1}. Next we are going to prove the following \emph{Lyapunov condition}:
\begin{equation}\label{Ljapunov}
	\text{for some } p >2: \qquad
	\lim_{N\to\infty} \ \sum_{i=1}^{M_N} \,
	\bbE \Big[ \big|   X_{N,i} \big|^p \Big] \,=\, 0 \,,
\end{equation}
which implies Lindeberg's condition \eqref{eq:Lindeberg-condition}
since 
\begin{equation*}
\bbE \big[ \big(X_{N,i}\big)^2 \, \ind_{\{|X_{N,i}| > \epsilon\}} \big]
	\le \bbE \Bigg[ \frac{|X_{N,i}|^p}{|X_{N,i}|^{p-2}} \, \ind_{|X_{N,i}| > \epsilon} \} \Bigg] \le  \,\frac{\bbE \big[ \big| X_{N,i}\big|^p \big]}{\epsilon^{p-2}} \,.
\end{equation*}

\smallskip

To obtain \eqref{Ljapunov}, we apply the hypercontractivity bound 
\eqref{eq:hypercontractivity} to $X_{N,i}$, see \eqref{eq:XNi}, to get
\begin{equation} \label{eq:hyperba}
	\bbE \Big[ \big|   X_{N,i} \big|^p \Big]^{\frac{2}{p}} \,\le\, 
	\sum_{\substack{A \subset \mathbb{B}_i \\ 
	|A| \le K_N}} (p-1)^{|A|} \, q_N(A)^2 
	\,\le\,
	 (p-1)^{K_N} \, \sigma_N^2(\bbB_i) \,,
\end{equation}
where we recall that $\sigma_N^2(\bbB_i) = \sum_{A \subset \mathbb{B}_i} q_N(A)^2$.
Then we  can write, for any $p > 2$,
\begin{equation} \label{eq:seethedisc}
\begin{split}
	\sum_{i=1}^{M_N} \, \bbE \Big[ \big|   X_{N,i} \big|^p \Big]	
	&\,\le\, \bigg( \max_{i=1,\ldots,M_N}  \bbE \Big[ \big|   X_{N,i} \big|^p \Big]
	\bigg)^{1-\frac{2}{p}}
	\ \sum_{i=1}^{M_N} \bbE \Big[ \big|   X_{N,i} \big|^p \Big]^{\frac{2}{p}}  \\
	&\,\le\, \Bigg\{ (p-1)^{p K_N} \,
	\Big( \max_{i=1,\ldots,M_N}  \sigma_N^2(\bbB_i) \Big)^{p-2} \Bigg\}^{\frac{1}{2}}
	\ \sum_{i=1}^{M_N} \, \sigma_N^2(\bbB_i) \,.
\end{split}
\end{equation}
If we fix $p=3$, the term in brackets vanishes as $N\to\infty$
by our choice \eqref{eq:KN} of $K_N$.
The last sum 
converges to $\sigma^2$ as $N\to\infty$,
see \eqref{eq:hyp3+}, hence it is uniformly bounded.
This completes the proof of \eqref{Ljapunov}.

\subsection{Switching to Gaussian random variables}
\label{sec:Lindeberg-principle}

We finally complete the proof of Theorem~\ref{th:main} by justifying the preliminary step:
we show that replacing the random variables $(\eta^N_t)_{t\in\bbT}$  in \eqref{eq:polychaosgen}
by standard Gaussians
\emph{does not change the asymptotic distribution of $X_N$}. More precisely,
if $(\hat\eta_t)_{t\in\bbT}$ are independent $\cN(0,1)$ and we set
\begin{equation}\label{eq:polychaosgen2}
	\hat X_N = \sum_{A \subset \mathbb{T}} q_N (A)  \,\hat\eta (A) \,, \qquad \text{with} \qquad
	\hat\eta (A) := \prod_{t \in A} \hat\eta_t \,,
\end{equation}
it suffices to show that for every bounded and smooth $f: \R \to \R$ we have
\begin{equation}\label{eq:goalconvdist}
	\lim_{N\to\infty} \ \big| \,\bbE[ f(X_N)] \,-\, \bbE[ f(\hat X_N) ] \, \big| \,=\, 0  \,.
\end{equation}
Indeed, since $\hat X_N \overset{d}{\to} \cN(0,\sigma^2)$
by the first part of the proof, \eqref{eq:goalconvdist} implies $X_N \overset{d}{\to} \cN(0,\sigma^2)$.

\smallskip

We exploit the Lindeberg principle \cite[Theorem 2.6]{CSZ17a},
which generalizes \cite{MOO10}, to show that $\bbE[ f(X_N)]$ is close to
$\bbE[ f(\hat X_N)]$. 
Let us fix $f:\R \to \R$ of class $C^3$ with
\begin{equation}
	C_f := \max \{ \| f' \|_{\infty}, \| f'' \|_{\infty}, \| f''' \|_{\infty} \} < \infty \,.
\end{equation}
For $L > 0$, denote by $m_2^{>L}$ the second moment tail of the random
variables $\eta^N_t$ and $\hat\eta_t$:
\begin{equation}\label{unifint}
	m_2^{>L} \,:=\, \sup_{N\in\N, \, t \in \mathbb{T}} \  \max 	\Big\lbrace \bbE \big[|\eta^N_t|^2 \ind_{|\eta^N_t|>L}\big]\,, \, \bbE \big[|\hat\eta_t|^2 \ind_{|\hat\eta_t|>L}\big]\Big\rbrace
	 \,.
\end{equation}
Let $\mathsf{C}_{{X}_N^{\le K }}$, $\mathsf{C}_{{X}_N^{> K }}$ be
the second moments of $X_N$ truncated to chaos of order $\le K$ and $> K$:
\begin{equation}\label{secmomtruncated}
	\mathsf{C}_{{X}_N^{\le K }} \,:=\, \sum_{\substack{A \subset \mathbb{T}\\
				|A| \le K }} q_N(A)^2 \,, \qquad
	\mathsf{C}_{{X}_N^{> K }} \,:=\, \sum_{\substack{A \subset \mathbb{T}\\
				|A| > K }} q_N(A)^2 \,.
\end{equation}
Finally, define the \emph{influence} of the variable $t \in \mathbb{T}$ on $X_N$ 
by\footnote{Note that we can write 
$\text{Inf}_t[X_N] =
\bbE \left[ \mathbb{V}ar \left[ X_N(\eta) | (\eta^N_s)_{s \in \mathbb{T}\setminus {t} }\right]\right]$.}
\begin{equation}\label{eq:influence}
	\text{Inf}_t[X_N] \,:=\, 
	\sum_{\substack{A \subset \mathbb{T}\\ A \ni t }} q_N(A)^2 \,.
\end{equation} 
By \cite[Theorem 2.6]{CSZ17a},
for any $L > 0$ such that $m_2^{>L} \le \frac{1}{4}$
and for every $K \in \N$ we have
\begin{equation}\label{Lindebergprinc}
	\begin{aligned} 
	\big| \bbE [ f(X_N)] - \bbE [ f(\hat X_N)] \big| 
	\,\le\,  C_f \ \bigg\{ & \, 2 \sqrt{\mathsf{C}_{{X}_N^{> K }} } 
	\,+\, 16 K^2 \, \mathsf{C}_{{X}_N^{\le K }} \, m_2^{>L}\\
	& \qquad \,+\, 70^{K +1} \, \mathsf{C}_{{X}_N^{\le K }} \, L^{3K} \, 
	\max_{t\in \mathbb{T}} \, \sqrt{\text{Inf}_t[X_N] }  \,  \bigg\} \,.
	\end{aligned}
\end{equation}

It remains to show that the r.h.s.\ of this expression is small as $N\to\infty$, to prove
\eqref{eq:goalconvdist}.
We fix any $\epsilon > 0$ and we argue as follows:
\begin{itemize}
\item by assumption \eqref{hyp2},
we can choose $K = K_\epsilon$ such that
$\limsup_{N\to\infty} \mathsf{C}_{{X}_N^{> K }} \le \epsilon$;

\item by assumption \eqref{hyp1},
for any $K\in\N$ we can bound
$\limsup_{N\to\infty} \, \mathsf{C}_{{X}_N^{\le K }} \,\le\, \sigma^2$;

\item by assumption \eqref{eq:ui}, we can choose $L = L_\epsilon$ such that
$m_2^{>L_\epsilon} \le \epsilon / (K_\epsilon^2 \, \sigma^2)$;

\item finally, we show below that
\begin{equation}\label{eq:lastclaim}
	\limsup_{N\to\infty} \ \max_{t\in \mathbb{T}} \, \sqrt{\text{Inf}_t[X_N] } \,=\, 0\,.
\end{equation}
\end{itemize}
As a consequence, when we plug $K=K_\epsilon$ and $L=L_\epsilon$ in \eqref{Lindebergprinc}
and we let $N\to\infty$, we get
\begin{equation*}
\begin{split}
	\limsup_{N\to\infty} \ \big| \bbE [ f(X_N)] - \bbE [ f(\hat X_N)] \big| 
	\,\le\,  C_f \, \big\{ \, 2\,\sqrt{\epsilon} \,+\, 16 \, \epsilon \big\} \,,
\end{split}
\end{equation*}
from which \eqref{eq:goalconvdist} follows because $\epsilon > 0$ is arbitrary.

\smallskip

It only remains to prove \eqref{eq:lastclaim}.
By assumption there are disjoint boxes $\bbB_1, \ldots, \bbB_{M_N} \subset \bbT$,
with $M_N\to\infty$, such that 
relation \eqref{eq:hyp3+} holds. In particular, recalling also \eqref{eq:sigmaB} and  \eqref{hyp1},
it follows that \emph{subsets $A \subset \bbT$
not contained in any of the boxes $\bbB_i$ give a negligible contribution}:
\begin{equation}\label{eq:Delta}
	\Delta_N \,:=\, \sum_{\substack{A \subset 	\mathbb{T}:\\
	A \nsubset \bbB_i \, \forall i=1,\ldots,M_N}} q_N(A)^2
	\,=\,\sigma_N^2(\bbT) \,-\,
	\sum_{i=1}^{M_N} \sigma_N^2(\bbB_i)
	\,\xrightarrow[N\to\infty]{}\, 0 \,.
\end{equation}
Recall now the definition \eqref{eq:influence} of $\text{Inf}_t[X_N]$.
Fix $t\in\bbT$ and a subset $A \subset \bbT$ which contains~$t$,
i.e.\ $A \ni t$. We distinguish two cases:
\begin{itemize}
\item if $t \not\in \bbB_i$ for all $i=1,\ldots, M_N$, then
$A \ni t$ implies $A \nsubset \bbB_i$ for all
$i=1,\ldots, M_N$, hence by \eqref{eq:Delta}
we can bound $\text{Inf}_t[X_N] \le \Delta_N$;

\item if $t \in \bbB_j$ for some (necessarily unique) $j=1,\ldots, M_N$,
then $A \ni t$ implies that either $A \subset \bbB_j$, or 
$A \nsubset \bbB_i$ for all $i=1,\ldots, M_N$ (we cannot have
$A \subset \bbB_i$ for some $i\ne j$), hence by \eqref{eq:sigmaB} and
\eqref{eq:Delta} we can bound
$\text{Inf}_t[X_N] \le \sigma_N^2(\bbB_j) + \Delta_N$.
\end{itemize}
It follows that
\begin{equation*}
	\max_{t\in\bbT} \ \text{Inf}_t[X_N] \,\le\,
	\max_{j=1,\ldots, M_N} \ \sigma_N^2(\bbB_j) \,+\, \Delta_N \,,
\end{equation*}
hence \eqref{eq:lastclaim} follows by \eqref{eq:hyp3+} and \eqref{eq:Delta}.
The proof of Theorem~\ref{th:main} is complete.\qed

\section{Proof of Theorem~\ref{th:singular}}
\label{sec:singular}

\subsection{Preparation}
We need to show that
\begin{equation*}
	(\dot{W}_N, \Xi_N) \,\overset{\cD}{\Longrightarrow}\,
	\Big(\dot{W}, \, \sqrt{c_{\hat\beta}^2-1} \,\dot{W}' \Big) \,,
\end{equation*}
that is, for any fixed $\psi \in C^\infty_c([0,1]\times\R^2)$ we have
\begin{equation}\label{tesiL}
	\big( \langle\dot{W}_N,\psi\rangle ,\,
	\langle \Xi_N, \psi\rangle \big) \,\xrightarrow[]{\ d \ }\,
	\cN\big(0, \|\psi\|_{L^2}^2 \, \Sigma_{\hat\beta} \big) \qquad \text{where} \qquad
	\Sigma_{\hat\beta} = \left(\begin{matrix}
		1 & 0 \\
		0 &  c_{\hat\beta}^2-1 
	\end{matrix}\right) \,.
\end{equation}
By the Cram\'er-Wold device \cite[Theorem~29.4]{Bil95}, 
it suffices to show that for all $\lambda,\mu\in\R$
\begin{equation}\label{goalL}
	X_N := \mu \, \langle \dot{W}_N,\psi \rangle 
	+ \lambda \, \langle \Xi_N,\psi \rangle \,\xrightarrow[]{\ d \ }\, 
	\cN \Big( 0,
	\sigma^2 := \, \| \psi\|_{L^2}^2 \big(\mu^2  + \lambda^2\, (c^2_{\hat{\beta}}-1) \big) \, \Big) \,.
\end{equation}
To this purpose we are going to apply Theorem~\ref{th:main}.

Recall the definitions \eqref{eq:apprwhite} and \eqref{eq:LN}
of $\dot{W}_N$ and $\Xi_N$ (see also \eqref{eq:JN}), we can write
\begin{equation}
\begin{split} \label{integralC}
	X_N &= 
	N \, \int\limits_{(0,1] \times\R^2}
	\psi(t,x) \, \eta_N\big(\lfloor Nt \rfloor, \lfloor \sqrt{N}x \rfloor\big)  \,
	\Big\{ \mu + \lambda \big(Z_N^{\beta_N}(\lfloor Nt \rfloor, \lfloor \sqrt{N}x \rfloor)-1\big) \Big\}\, 
	\dd t \, \dd x \\
	&= 
	\frac{1}{N} \, \int\limits_{(0,N] \times\R^2}
	\psi\big(\tfrac{t}{N}, \tfrac{x}{\sqrt{N}}\big) \, \eta_N\big(\lfloor t \rfloor, \lfloor x \rfloor\big)  \, 
	\Big\{ \mu + \lambda \big(Z_N^{\beta_N}(\lfloor t \rfloor, \lfloor x \rfloor)-1\big) \Big\} \, 	
	\dd t \, \dd x \,.
\end{split}
\end{equation}
Let us define $\overline{\psi}_N : \N \times \Z^2 \to \R$ as the average of 
$\psi \big(\frac{\cdot}{N},\frac{\cdot}{\sqrt{N}} \big)$ over cubes:
\begin{equation} \label{eq:psibar}
	\overline{\psi}_N (n,z):= 
	\int\limits_{(n-1,n] \times\{ (z_1-1, z_1] \times ( z_2-1, z_2]\}} 
	\psi \big(\tfrac{t}{N},\tfrac{x}{\sqrt{N}} \big) \, \dd t \, \dd x \qquad
	\text{for} \quad (n,z) \in \N\times\Z^2 \,.
\end{equation}
Recalling the polynomial chaos expansion \eqref{eq:Zpoly} of $Z_N^{\beta_N}(m,z)$, we 
can rewrite $X_N$ as follows:
\begin{equation*}
\begin{split}
	X_N = \frac{1}{N} \sum_{n_0=1}^N & \, \sum_{x_0 \in \Z^2} \,
	\overline{\psi}_N (n_0,x_0) \, \eta_N(n_0,x_0)  \\
	& \quad \Bigg\{ \mu \, + \,
	\lambda \,\sum_{k=1}^\infty \, (\sigma_N)^k \sum_{\substack{n_0<n_1<\ldots<n_k\le N\\ 
	x_0,x_1,\ldots,x_k \in \Z^2}} \,
	\prod_{j=1}^k \, q_{n_j-n_{j-1}}(x_j-x_{j-1}) \,\eta_N(n_j,x_j) \Bigg\} \,.
\end{split}
\end{equation*}
Renaming $(n_0, \ldots, n_k)$ as $(n_1, \ldots, n_{k+1})$
and similarly $(x_0, \ldots, x_k)$ as $(x_1, \ldots, x_{k+1})$, and subsequently
renaming $k+1$ as $k$, we obtain the compact expression
\begin{gather} \label{eq:ix}
	X_N = \frac{1}{N} \sum_{k=1}^\infty \, (\sigma_N)^{k-1}
	\, \sum_{\substack{0<n_1<\ldots<n_k\le N\\ 
	x_1,\ldots,x_k \in \Z^2}} \, f_N(n_1,x_1, \ldots, n_k,x_k) \,
	\prod_{j=1}^k \eta_N(n_j,x_j) \,,
\end{gather}
where we set
\begin{equation} \label{eq:effe}
	f_N(n_1,x_1, \ldots, n_k,x_k) \,:=\,
	\big\{ \mu \, \ind_{\{k=1\}} + \lambda \, \ind_{\{k\ge 2\}} \big\} \,
	\overline{\psi}_N (n_1,x_1) \,
	\prod_{j=2}^k  q_{n_j-n_{j-1}}(x_j-x_{j-1})  \,.
\end{equation}

In conclusion, we can write $X_N = \sum_{A \subset \bbT} q_N(A) \, \eta^N(A)$
as in \eqref{eq:polychaosgen}-\eqref{eq:polychaosgen-k},
with the following correspondences:
\begin{itemize}
	\item the index set is $\bbT := \N \times \Z^2$; 
	\item the random variables $\eta^N_t = \eta_N(m,z)$, for $t = (m,z) \in \bbT$,
	are defined in \eqref{eq:eta}: they satisfy 
	\eqref{eq:disordergen} by construction, while they satisfy \eqref{eq:ui}
	because $\sup_N \bbE[|\eta_N(m,z)|^{p}] < \infty$
for all $p < \infty$ by \eqref{eq:disorder}
(see \cite[eq. (6.7)]{CSZ17a});
	\item the kernel $q_N(A)$,
	for $A:= \{t_1,\ldots,t_{k}\}=\{(n_1,x_1),\ldots,(n_{k}, x_{k})\} \subseteq \bbT$,
	is
\begin{equation*}
	\begin{split}
		q_N(A) = \frac{1}{N} \, (\sigma_N)^{k-1} \, 
		f_N(n_1,x_1, \ldots, n_k,x_k) \, \ind_{ \{0<n_1<\ldots<n_k \le N\} }  \,.
	\end{split}
\end{equation*}
\end{itemize}
By Theorem~\ref{th:main}, to prove
$X_N \overset{d}{\to} \cN(0,\sigma^2)$ as in \eqref{goalL}, we check the following conditions.
\begin{enumerate}
	\item\label{it:sing1} \emph{Limiting second moment:}
	we need to prove that $\lim_{N\to\infty} \bbE[X_N^2] = \sigma^2$.

	\item\label{it:sing2} \emph{Subcriticality}:
	we need to show that
\begin{equation} \label{eq:fromwhich0}
		\lim_{K\to \infty} \ \limsup_{N \rightarrow \infty}  \
		\sum_{\substack{A \subset \mathbb{T}\\ |A| > K}}q_N(A)^2=0 \,.
\end{equation}

	\item\label{it:sing3}
	\emph{Spectral localization}: for any $M,N \in \N$ 
	we define the disjoint subsets
	\[ \bbB_j := \big( \tfrac{j-1}{M}N, \tfrac{j}{M}N \big]  \, \times \, \Z^2 \qquad
	\text{for } j=1,\ldots, M \,,
	\]
	and, recalling that $\sigma_N^2(\bbB_j) := \sum_{A \subset \bbB_j} q_N(A)^2$,
	we need to show that
\begin{equation} \label{eq:needto}
\begin{split}
	\lim_{M\to\infty} \ \sum_{j = 1}^M \
	\lim_{N \rightarrow \infty}  \ \sigma^2_N (\mathbb{B}_j) \,=\,
	\sigma^2 \qquad \text{and} \qquad
	\lim_{M \rightarrow \infty} \ \Big\{ \max_{j =
		1, \ldots, M} \ \limsup_{N \rightarrow \infty} \  
		\sigma^2_N (\mathbb{B}_j) \Big\} \,=\, 0 \,.
\end{split}
\end{equation}
\end{enumerate}

\smallskip
\subsection{Proof of \eqref{it:sing2}.}
We need to prove \eqref{eq:fromwhich0}.
For $K \ge 1$ we can write, by \eqref{eq:ix}-\eqref{eq:effe},
\begin{equation} \label{eq:sumsquares}
	\sum_{\substack{A \subset \mathbb{T}\\ |A| > K}}q_N(A)^2
	= \frac{\lambda^2}{N^2} \sum_{k> K} (\sigma_N^2)^{k-1}
	\sum_{\substack{0<n_1<\ldots<n_k\le N\\ 
	x_1,\ldots,x_k \in \Z^2}} 
	\overline{\psi}_N (n_1,x_1)^2 \,
	\prod_{j=2}^k  q_{n_j-n_{j-1}}(x_j-x_{j-1})^2 .
\end{equation}
We can enlarge the 
sums to $0 < m_j := n_j-n_{j-1}\le N$
and change variables $y_j := x_j - x_{j-1}$,
for $j=2,\ldots, k$, to get the upper bound
\begin{equation} \label{eq:arg1}
\begin{split}
	\sum_{\substack{A \subset \mathbb{T}\\ |A| > K}}q_N(A)^2
	& \le \frac{\lambda^2}{N^2} \sum_{k> K} (\sigma_N^2)^{k-1}
	\sum_{\substack{0<n_1\le N \\ x_1 \in \Z^2}} \,
	\overline{\psi}_N (n_1,x_1)^2 \,
	\prod_{j=2}^k \Bigg\{ \sum_{\substack{0 < m_j \le N \\ y_j \in \Z^2}}
	\ q_{m_j}(y_j)^2 \Bigg\} \\
	& =  \lambda^2 \, \Bigg\{
	\frac{1}{N^2} \sum_{\substack{0<n_1\le N \\ x_1 \in \Z^2}} \,
	\overline{\psi}_N (n_1,x_1)^2 \Bigg\} 
	 \, \frac{(\sigma_N^2 \, R_N)^{K}}{1-\sigma_N^2 \, R_N} \,,
\end{split}
\end{equation}
where we used $\sum_{0 < m \le N} \sum_{y\in\Z^2} q_m(y)^2
= \sum_{0 < m \le N} u_m = R_N$, see \eqref{eq:uN}-\eqref{eq:RN}, and we remark that $\sigma_N^2R_N<1$ for $N$ large enough, because
$\sigma_N^2 \sim \hat\beta^2 / R_N$, see \eqref{eq:sub-critical}, and $\hat{\beta}<1$. Then, by Riemann sum approximation, from \eqref{eq:psibar} we get
\begin{equation} \label{eq:arg2}
	\limsup_{N\to\infty} \sum_{\substack{A \subset \mathbb{T}\\ |A| > K}}q_N(A)^2
	\,\le\, \lambda^2 \, \bigg\{ \int_{[0,1]\times\R^2} \psi(t,x)^2 \, \dd t \, \dd x
	\bigg\} \, \frac{(\hat\beta^2)^K}{1-\hat\beta^2} 
	\,=\, \lambda^2 \, \|\psi\|_{L^2}^2
	\, \frac{(\hat\beta^2)^K}{1-\hat\beta^2} \,,
\end{equation}
from which \eqref{eq:fromwhich0} follows.

\smallskip

\subsection{Proof of \eqref{it:sing1} and \eqref{it:sing3}}
We are going to show that
for all $M\in\N$ and $j \in \{1,\ldots, M\}$
\begin{equation} \label{isecmom}
	\lim_{N \to \infty} \sigma_N^2(\bbB_j) = \big( \mu^2 \,+\,
	\lambda^2 (c^2_{\hat{\beta}}-1) \big) \, 
	\int_{(\frac{j-1}{M},\frac{j}{M}]\times\R^2} 
	\psi(t,x)^2 \, \dd t \, \dd x\,.
\end{equation}
Note that this proves \eqref{eq:needto} and also (for $j=M=1$)
$\lim_{N\to\infty} \bbE[X_N^2] = \sigma^2$, see \eqref{goalL}.

To compute $\sigma_N^2(\bbB_j) := \sum_{A \subset \bbB_j} q_N(A)^2$
we first consider the contribution of sets $A \subset \bbB_j$
with $|A| = 1$, that is $A = \{(n_1,x_1)\}$. Since $f_N(n_1,x_1) = \mu \,
\overline{\psi}_N(n_1,x_1)$, see \eqref{eq:effe}, we get
\begin{equation*}
	\sum_{A \subset \bbB_j ,\, |A|=1} q_N(A)^2
	= \frac{\mu^2}{N^2} \, 
	\sum_{\substack{\frac{j-1}{M}N<n_1\le\frac{j}{M} N \\ x_1 \in \Z^2}} \,
	\overline{\psi}_N (n_1,x_1)^2 \,\xrightarrow[]{\,N\to\infty\,}\,
	\mu^2 \, \int_{(\frac{j-1}{M},\frac{j}{M}]\times\R^2} 
	\psi(t,x)^2 \, \dd t \, \dd x \,,
\end{equation*}
by Riemann sum approximation. Note that this matches with the first term in \eqref{isecmom}.

We next focus on sets $A \subset \bbB_j$
with $|A| > 1$. Note that $\sum_{A \subset \bbB_j ,\, |A|>1} q_N(A)^2$
is given by \eqref{eq:sumsquares} with $K=1$ and with the sum restricted
to $\frac{j-1}{M}N<n_1 < \ldots < n_k \le\frac{j}{M} N$.
Then, arguing as in \eqref{eq:arg1}, we obtain an analogue of \eqref{eq:arg2}:
\begin{equation*}
	\limsup_{N\to\infty} \sum_{A \subset \bbB_j ,\, |A|>1} q_N(A)^2
	\,\le\, \lambda^2 \, \Bigg\{ \int_{(\frac{j-1}{M},\frac{j}{M}]\times\R^2} 
	\psi(t,x)^2 \, \dd t \, \dd x
	\Bigg\} \, \frac{\hat\beta^2}{1-\hat\beta^2}  \,,
\end{equation*}
which agrees with the second term in \eqref{isecmom}
because $\frac{\hat\beta^2}{1-\hat\beta^2} = c_{\hat\beta}^2-1$,
see \eqref{eq:cdef}. To complete the proof, it suffices to prove a matching lower bound, that is
\begin{equation} \label{eq:matching}
	\liminf_{N\to\infty} \sum_{A \subset \bbB_j ,\, |A|>1} q_N(A)^2
	\,\ge\, \lambda^2 \, \Bigg\{ \int_{(\frac{j-1}{M},\frac{j}{M}]\times\R^2} 
	\psi(t,x)^2 \, \dd t \, \dd x
	\Bigg\} \, \frac{\hat\beta^2}{1-\hat\beta^2}  \,.
\end{equation}

Let us fix $H \in \N$ large, such that $\frac{1}{H}<\frac{1}{M}$.
Starting from the expression \eqref{eq:sumsquares} for $K=1$ and 
with $\frac{j-1}{M}N<n_1 < \ldots < n_k \le\frac{j}{M} N$,
we get a lower bound 
by the following restrictions:
\begin{equation*}
	1 < k \le H \,, \qquad \tfrac{j-1}{M}N < n_1 \le \big(\tfrac{j}{M} - \tfrac{1}{H}\big)N \,, \qquad
	0 < n_j - n_{j-1} \le \tfrac{1}{H^2} N \quad \forall j=2,\ldots, k \,,
\end{equation*}
which ensure that $n_k \le n_1 + \sum_{j=2}^k (n_j-n_{j-1}) \le 
(\frac{j}{M} - \frac{1}{H})N + H  \frac{1}{H^2} N \le \frac{j}{M} N$
as required.
Then, similarly to \eqref{eq:arg1}, we get the following lower bound
on $\sum_{A \subset \bbB_j ,\, |A|>1} q_N(A)^2$:
\begin{equation} \label{eq:arglast}
\begin{split}
	& \frac{\lambda^2}{N^2} \sum_{k=2}^H (\sigma_N^2)^{k-1}
	\sum_{\substack{\frac{j-1}{M}<n_1\le(\frac{j}{M}-\frac{1}{H}) N \\ x_1 \in \Z^2}} \,
	\overline{\psi}_N (n_1,x_1)^2 \,
	\prod_{j=2}^k \Bigg\{ \sum_{\substack{0 < m_j \le \frac{1}{H^2}N \\ y_j \in \Z^2}}
	\ q_{m_j}(y_j)^2 \Bigg\} \\
	& \qquad =   \Bigg\{
	\frac{\lambda^2}{N^2} 
	\sum_{\substack{\frac{j-1}{M}<n_1\le(\frac{j}{M}-\frac{1}{H}) N \\ x_1 \in \Z^2}} \,
	\overline{\psi}_N (n_1,x_1)^2 \Bigg\} 
	 \, \frac{\sigma_N^2 \, R_{N/H^2} - (\sigma_N^2 \, R_{N/H^2})^{H}}{1-\sigma_N^2 
	 \, R_{N/H^2}} \,,
\end{split}
\end{equation}
where we recall that $\sum_{k=2}^H x^{k-1} = \frac{x-x^H}{1-x}$ for $|x|<1$.
Since $R_{N/H^2} \sim R_N$ for fixed $H \in \N$, we have shown that
\begin{equation*}
	\liminf_{N \to \infty} \sum_{A \subset \bbB_j ,\, |A|>1} q_N(A)^2 
	\ge  
	\lambda^2 \, \bigg\{ \int_{(\frac{j-1}{M},\frac{j}{M}-\frac{1}{H}]\times\R^2} 
	\psi(t,x)^2 \, \dd t \, \dd x
	\bigg\} \, \frac{\hat\beta^2 - (\hat\beta^2)^H}{1-\hat\beta^2} \,.
\end{equation*} 
We can finally take the limit $H\to\infty$ to see that \eqref{eq:matching} holds.\qed

\section{Proof of Theorem~\ref{th:apprZ}} 
\label{sec:apprZ}

The proof is organised in four parts: we give different approximations
of the partition function $Z_N^{\beta_N}$ and of its logarithm, which will lead us to the proof
of our goal~\eqref{eq:apprZ}. Let us present a general overview of the strategy.

\medskip
\noindent
\emph{Part 1 (record times).}
Let us define a ``constrained version''
$X_{N,[a,b;b']}^{\dom}(x,z;z')$ of $X_N^{\dom}$ from \eqref{eq:XNdom},
where we fix  $(n_0, n_1; n_k) = (a, b; b')$ and 
$(x_0, x_1; x_k) = (x,z; z')$:
\begin{equation}\label{XNdom}
\begin{split}
	& X_{N,[a,b;b']}^{\dom}(x,z;z') \,:=\, \sum_{k=1}^\infty (\sigma_N)^k
	\, q_{b-a}(z-x) \, \eta_N(b,z) \times \\
	& \qquad\qquad \times \sum_{\substack{b =: n_1 < n_2 < \ldots < n_{k-1} < n_k  = : b'\\
	\max\{ n_2-n_1, \ldots, n_k - n_{k-1} \} \le b}}
	\ \sum_{\substack{x_1 = z, \, x_k = z',\\
	x_2, \ldots, x_{k-1} \in \Z^2}} \
	\prod_{i=2}^k q_{n_i-n_{i-1}}(x_i-x_{i-1}) \, \eta_N(n_i, x_i) \,.
\end{split}
\end{equation}
(Note that if $b=b'$ only the terms $k=1$ contributes to the sum ---
and we must have $z=z'$, otherwise the sum vanishes ---
while if $b < b'$ only the terms $k \ge 2$ give a contribution.)

\smallskip

We first show that the partition function $Z_{N}^{\beta_N}$ in \eqref{eq:Zpoly}
can be written as a concatenation of
products of $X_{N,[a,b;b']}^{\dom}(x,z;z')$'s corresponding to suitable
\emph{record times}, see Figure~\ref{fig1}.
The next result is proved in subsection~\ref{sec:proof-first-step}.

\begin{lemma}[Record times]\label{th:first-step}
The following equality holds, with $(b_0',z_0') := (0,0)$:
\begin{equation}\label{ZconXdom}
\begin{split}
	Z_{N}^{\beta_N} & \,=\, 1 + \sum_{\ell=1}^\infty \
	\sum_{\substack{0 < b_1 \le b'_1 < \ldots < b_\ell \le b'_\ell \le N:\\
	b_i - b'_{i-1} > b_{i-1} \, \forall i=2,\ldots, \ell}}
	\ \sum_{\underline{z}, \underline{z}' \in (\Z^2)^\ell} \
	\prod_{i=1}^\ell X_{N,[b'_{i-1},b_i;b_i']}^{\dom}(z'_{i-1},z_i;z'_i) \,,
\end{split}
\end{equation}
where we use the shortcuts
$\underline{z} = (z_1, \ldots, z_\ell)$ and $\underline{z}' = (z'_1, \ldots, z'_\ell)$.
\end{lemma}

\smallskip
\medskip
\noindent
\emph{Part 2 (coarse-graining and diffusive approximation).}
We fix a large parameter $M\in\N$ and we
define an approximation $Z_{N,M}^{(\diff)}$ of
the partition function $Z_{N}^{\beta_N}$ from 
\eqref{ZconXdom}, as follows:\footnote{Heuristically, 
these are good approximations because
the main contribution to \eqref{ZconXdom}
will be shown to come from $b'_{i-1} \approx N^{\alpha'_{i-1}}$ and $b_i \approx N^{\alpha_i}$
with $\alpha'_{i-1}<\alpha_i$, hence $b'_{i-1} \ll b_i$.}
\begin{enumerate}
\item we set $b'_{i-1} = 0$, $z'_{i-1}=0$ in each $X_{N,[b'_{i-1},b_i;b_i']}^{\dom}(z'_{i-1},z_i;z'_i)$;
\item we impose that each pair $b_i \le b'_i$ belongs to the same
interval $(N^{\frac{j-1}{M}}, N^{\frac{j}{M}}]$, for some $j=1,\ldots, M$, and we ignore
the constraint $b_i - b'_{i-1} > b_{i-1}$.
\end{enumerate}
This yields the following definition of $Z_{N,M}^{(\diff)}$:
\begin{equation}\label{Z2multform}
	\begin{split}
	Z_{N,M}^{(\diff)}
	\, := \,1 \,+\, \sum_{\ell=1}^\infty \, \sum_{1\le j_1 < \ldots < j_\ell \le M} \, 
	\prod_{i=1}^\ell X_{N,M}^{\dom}(j_i)
	\,= \,\prod_{j=1}^M \, \left( 1 + X_{N,M}^{\dom}(j) \right) \,,
\end{split}
\end{equation}
where we set
\begin{equation}\label{Xdomj}
	X_{N,M}^{\dom}(j) := \sum_{b\le b' \in (N^{\frac{j-1}{M}},N^{\frac{j}{M}} ] }
	\ \sum_{z,z' \in \Z^2} \ X_{N,[0,b;b']}^{\dom}(0,z;z') \qquad
	\text{for} \ j=1,\ldots, M \,.
\end{equation}

\smallskip

We prove that $Z_{N,M}^{(\diff)}$ is close
to $Z_{N}^{\beta_N}$ in $L^2$ for $N \gg M \gg 1$, in the following sense.

\begin{lemma}[Coarse-graining and diffusive approximation]\label{th:second-step}
The following holds:
\begin{equation}\label{eq:second-step}
	\limsup_{M\to\infty} \ \limsup_{N\to\infty} \
	\big\| Z_N^{\beta_N} - Z_{N,M}^{(\diff)} \big\|_{L^2} = 0 \,.
\end{equation}
\end{lemma}

\noindent
The proof of this result is given in subsection~\ref{sec:proof-second-step} below.

\smallskip
\medskip
\noindent
\emph{Part 3 (log approximation).}
The product form of $Z_{N,M}^{(\diff)}$ in \eqref{Z2multform} is especially suitable
to take the logarithm.
We thus prove a preliminary version of our goal~\eqref{eq:apprZ}, where
we replace $\log Z_N^{\beta_N}$ by $\log Z_{N,M}^{(\diff)}$ (and convergence in $L^2$
by convergence in probability).
To this purpose, we define the event
\begin{equation} \label{eq:ANM}
	A_{N,M}:=  \bigcap_{j=1}^M \big\{ |X_{N,M}^{\dom}(j)| \le \tfrac{1}{2} \big\}  \,,
\end{equation}
which ensures that $Z_{N,M}^{(\diff)}>0$,
see \eqref{Z2multform}.

\begin{lemma}[log approximation]\label{th:third-step}
Recall $X_N^\dom$ from \eqref{eq:XNdom}.
For any $\epsilon > 0$ we have
\begin{equation}\label{approxlogXdom}
	\lim_{M\to\infty} \
	\limsup_{N\to\infty} \  \bbP \Big( \big| \log Z_{N,M}^{(\diff)} - 
	\big\{ X_N^\dom - \tfrac{1}{2} \bbE[(X_N^\dom)^2] \big\} 
	\big| > \epsilon , \, A_{N,M} \Big)
	= 0 \,,
\end{equation}
for $A_{N,M} \subseteq \{Z_{N,M}^{(\diff)}>0\}$ defined in \eqref{eq:ANM}
(so that $\log Z_{N,M}^{(\diff)}$ is well-defined) which satisfies
\begin{equation}\label{PA_NM}
	\lim_{M \rightarrow \infty} 
	\ \limsup_{N \to\infty} \ \bbP \big( (A_{N,M})^c \big)  \,=\, 0 \,.
\end{equation}
\end{lemma}

\noindent
The proof of this result is given in subsection~\ref{sec:proof-third-step} below.

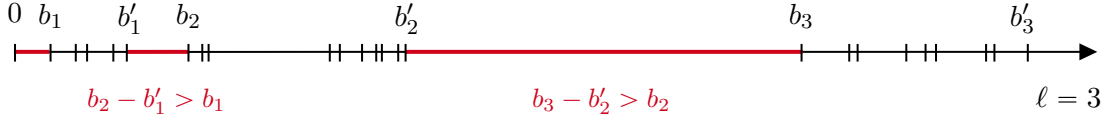
\begin{figure}
	\centering
	\tikzset{every picture/.style={line width=0.75pt}} 
	
	\begin{tikzpicture}[x=0.75pt,y=0.75pt,yscale=-1,xscale=1]
	
	\draw    (158.6,54.51) ;
	\draw [color={rgb, 255:red, 208; green, 2; blue, 27 }  ,draw opacity=1 ][line width=1.5]    (58.36,54.49) -- (76.12,54.49) ;
	\draw [color={rgb, 255:red, 0; green, 0; blue, 0 }  ,draw opacity=1 ]   (76.12,54.49) -- (107.48,54.49) ;
	\draw [color={rgb, 255:red, 208; green, 2; blue, 27 }  ,draw opacity=1 ][line width=1.5]    (114.59,54.51) -- (137.16,54.49) -- (145.77,54.49) ;
	\draw [color={rgb, 255:red, 0; green, 0; blue, 0 }  ,draw opacity=1 ]   (145.45,54.53) -- (221.52,54.49) ;
	\draw [color={rgb, 255:red, 208; green, 2; blue, 27 }  ,draw opacity=1 ][line width=1.5]    (253.99,54.49) -- (451.83,54.49) ;
	\draw [color={rgb, 255:red, 0; green, 0; blue, 0 }  ,draw opacity=1 ]   (451.83,54.49) -- (595.9,54.49) ;
	\draw [shift={(598.9,54.49)}, rotate = 180] [fill={rgb, 255:red, 0; green, 0; blue, 0 }  ,fill opacity=1 ][line width=0.08]  [draw opacity=0] (8.93,-4.29) -- (0,0) -- (8.93,4.29) -- cycle    ;
	\draw    (89.23,50.48) -- (89.23,59.35) ;
	\draw    (94.87,50.48) -- (94.87,59.35) ;
	\draw    (107.97,50.5) -- (107.97,59.37) ;
	\draw    (76.64,50.5) -- (76.64,59.37) ;
	\draw    (145.45,50.5) -- (145.45,59.37) ;
	\draw    (155.49,50.5) -- (155.49,59.37) ;
	\draw    (152.35,50.5) -- (152.35,59.37) ;
	\draw    (215.96,50.5) -- (215.96,59.37) ;
	\draw    (220.97,50.5) -- (220.97,59.37) ;
	\draw    (253.78,50.5) -- (253.78,59.37) ;
	\draw    (232.06,50.5) -- (232.06,59.37) ;
	\draw    (239.14,50.5) -- (239.14,59.37) ;
	\draw    (513.21,50.48) -- (513.21,59.35) ;
	\draw    (518.43,50.48) -- (518.43,59.35) ;
	\draw    (451.34,50.5) -- (451.34,59.37) ;
	\draw    (503.61,50.48) -- (503.61,59.35) ;
	\draw    (114.69,50.5) -- (114.69,59.37) ;
	\draw    (58.91,50.5) -- (58.91,59.37) ;
	\draw [color={rgb, 255:red, 0; green, 0; blue, 0 }  ,draw opacity=1 ]   (107.48,54.49) -- (114.59,54.51) ;
	\draw    (241.95,50.5) -- (241.95,59.37) ;
	\draw    (250.1,50.5) -- (250.1,59.37) ;
	\draw    (221.52,54.49) -- (253.99,54.49) ;
	\draw    (543.41,50.5) -- (543.41,59.37) ;
	\draw    (547.57,50.5) -- (547.57,59.37) ;
	\draw    (564.22,50.5) -- (564.22,59.37) ;
	\draw    (475.13,50.5) -- (475.13,59.37) ;
	\draw    (479.31,50.5) -- (479.31,59.37) ;
	
	\draw (53.6,28.74) node [anchor=north west][inner sep=0.75pt]  [font=\normalsize]  {$0$};
	\draw (68.97,28.65) node [anchor=north west][inner sep=0.75pt]  [font=\normalsize]  {$b_{1}$};
	\draw (108.46,28.65) node [anchor=north west][inner sep=0.75pt]  [font=\normalsize]  {$b_{1} '$};
	\draw (137.68,28.65) node [anchor=north west][inner sep=0.75pt]  [font=\normalsize]  {$b_{2}$};
	\draw (246.71,28.65) node [anchor=north west][inner sep=0.75pt]  [font=\normalsize]  {$b_{2} '$};
	\draw (443.44,28.65) node [anchor=north west][inner sep=0.75pt]  [font=\normalsize]  {$b_{3}$};
	\draw (553.79,28.62) node [anchor=north west][inner sep=0.75pt]  [font=\normalsize]  {$b_{3} '$};
	\draw (566.71,70.74) node [anchor=north west][inner sep=0.75pt]  [font=\normalsize]  {$\ell =3$};
	\draw (93.39,70.78) node [anchor=north west][inner sep=0.75pt]  [font=\small,color={rgb, 255:red, 208; green, 2; blue, 27 }  ,opacity=1 ]  {$b_{2} -b_{1} ' >b_{1}$};
	\draw (315.38,70.74) node [anchor=north west][inner sep=0.75pt]  [font=\small,color={rgb, 255:red, 208; green, 2; blue, 27 }  ,opacity=1 ]  {$b_{3} -b_{2} ' >b_{2}$};

	\end{tikzpicture}
	
	\caption{An example of the variables $b_i, b'_i$ in \eqref{ZconXdom}.
	These correspond to \emph{record times} which satisfy $b_i - b_{i-1}' > b_{i-1}$,
	see subsection~\ref{sec:proof-first-step}.}
	\label{fig1}
\end{figure}

\smallskip
\medskip
\noindent
\emph{Part 4 (final approximation).}
At last, we complete the proof of Theorem~\ref{th:apprZ}.
Our final goal \eqref{eq:apprZ} is a consequence of the next lemma,
where we prove convergence in probability and 
boundedness in $L^p$ for some $p>2$. 

\begin{lemma}[Final approximation]\label{th:fourth-step}
Recall $X_N^\dom$ from \eqref{eq:XNdom}.
For any $\epsilon > 0$ we have
\begin{equation}\label{eq:goal1}
	\lim_{N\to\infty} \, \bbP \big( \big| \log Z_N^{\beta_N} - 
	\big\{ X_N^\dom - \tfrac{1}{2} \bbE[(X_N^\dom)^2] \big\} 
	\big| > \epsilon \big) = 0 \,.
\end{equation}
Moreover, for some $p>2$ we have
\begin{equation} \label{eq:boundLp}
	\sup_{N\in\N} \ \bbE\big[ \big| \log Z_N^{\beta_N} \big|^p \big] < \infty \,, \qquad \
	\sup_{N\in\N} \ \bbE\big[ \big| X_N^\dom \big|^p \big] < \infty \,.
\end{equation}
\end{lemma}
\noindent
Notice that, once we have convergence in probability \eqref{eq:goal1}, to obtain 
convergence in $L^2$ it suffices to show \emph{uniform integrability of the squares of 
$\log Z_N^{\beta_N}$ and $X_N^\dom$}, which is in turn implied by boundedness in 
$L^p$ for some $p>2$, as in \eqref{eq:boundLp}.

Intuitively, we can deduce \eqref{eq:goal1} from \eqref{approxlogXdom} by
exploiting the approximation \eqref{eq:second-step}, but some care is needed
to handle the logarithm.

The proof of Lemma~\ref{th:fourth-step}, given in subsection~\ref{sec:proof-fourth-step},
concludes the proof of Theorem~\ref{th:apprZ}.\qed

\subsection{Proof of Lemma~\ref{th:first-step}}
\label{sec:proof-first-step}

We rewrite the sum over $n_1, \ldots, n_k$
in \eqref{eq:Zpoly} according to suitable \emph{record times}. The first record time is $n_{1}$;
the second  record time is the smallest $n_i$ for which the previous jump $n_i - n_{i-1}$
exceeds $n_1$; and so on. More precisely, the record times are $n_{j_1}, n_{j_2}, \ldots, n_{j_\ell}$
where we define $j_1 := 1$ and, assuming that $j_r < \infty$, we set
$j_{r+1} := \min\{i \in \{j_r+1, \ldots, k\}: \, n_i - n_{i-1} > n_{j_r}\}$, 
where we agree that $\min\emptyset := \infty$.
The number of record times is therefore
$\ell := \min\{r \ge 1: \ j_{r+1} = \infty\}$.

If we rename the record times as $b_r := n_{j_r}$, and
we also set $b_{r-1}' := n_{j_r-1}$, we have by construction
$b_2-b'_1 > b_1$ and, more generally, $b_i -b'_{i-1} > b_{i-1}$ for $i=2,\ldots, \ell$
(see Figure~\ref{fig1}).
If we name the corresponding space variables $z_r := x_{b_r}$ and $z'_{r-1} := x_{b'_{r-1}}$,
then we can rewrite \eqref{eq:Zpoly} equivalently as \eqref{ZconXdom},
with $X_{N,[a,b;b']}^{\dom}(x,z;z')$
defined in \eqref{XNdom}.\qed

\subsection{Proof of Lemma~\ref{th:second-step}}

\label{sec:proof-second-step}

The proof, which is long and structured, is based on explicit $L^2$ computations.
A key observation is that, by the expression \eqref{ZconXdom} for $Z_N^{\beta_N}$,
we can write
\begin{equation} \label{eq:reca}
\begin{split}
	\bbE \Big[ \big(Z_{N}^{\beta_N}\big)^2 \Big] = 1 + \sum_{\ell=1}^\infty 
	\ \sum_{\substack{0 < b_1 \le b'_1 < \ldots < b_\ell \le b'_\ell \le N:\\
	b_i - b'_{i-1} > b_{i-1} \, \forall i=2,\ldots, \ell}}
	\ \sum_{\underline{z}, \underline{z}' \in (\Z^2)^\ell} \
	\prod_{i=1}^\ell \bbE \Big[ \big( X^{\rm{dom}}_{N,[b_{i-1}',b_i;b_i']}(z_{i-1}',z_i;z'_i) 
	\big)^2 \Big] \,.
\end{split}
\end{equation}
To see why this holds, note that by \eqref{eq:Zpoly} we can write
\begin{equation} \label{eq:2mom}
	\begin{split}
	\bbE \big[ \big( Z_N^{\beta_N}  \big)^2 \big] = 1+ \sum_{k=1}^\infty (\sigma_N^2)^k 
	\sum_{\substack{0=:n_0 < n_1 < \ldots < n_k \le N \\
	x_0:=0, \ x_1, \ldots, x_k \in \Z^2 }} \
	\prod_{j=1}^k q_{n_j-n_{j-1}}(x_j-x_{j-1})^2 \,,
	\end{split}
\end{equation}
with $q_n(x) = \P(S_n = x \, | \, S_0=0)$, see \eqref{eq:qn}, and $\sigma_N$
as in \eqref{eq:eta}.
Similarly, by \eqref{XNdom},
\begin{equation} \label{eq:2momdom}
\begin{split}
	& \bbE \Big[ \big( X^{\rm{dom}}_{N,[a,b_;b']}(x,z; z') \big)^2 \Big]
	= \sum_{k=1}^\infty (\sigma_N^2)^k
	\, q_{b-a}(z-x)^2 \times  \\
	& \qquad\qquad\qquad\qquad \times 
	\sum_{\substack{b =: n_1 < n_2 < \ldots < n_{k-1} < n_k = b' \\
	\max\{ n_2-n_1, \ldots, n_k - n_{k-1} \} \le b}}
	\ \sum_{\substack{x_1 = z; \, x_k = z'\\
	x_2, \ldots, x_{k-1} \in \Z^2}} \
	\prod_{i=2}^k q_{n_i-n_{i-1}}(x_i-x_{i-1})^2 .
\end{split}
\end{equation}
When we plug \eqref{eq:2momdom} into \eqref{eq:reca}
we obtain \eqref{eq:2mom} by the same argument
in the proof of Lemma~\ref{th:first-step}, 
see subsection~\ref{sec:proof-first-step}, because the sum over
$n_j, x_j$ in \eqref{eq:2mom} can be rewritten in terms of record times,
which lead to the variables $b_r, b'_r$ and $z_r, z'_r$ in \eqref{eq:reca}.

\medskip

We now turn to the proof of \eqref{eq:second-step}. 
We will define two ``coarse-grained approximations'' $Z_{N, K,M}^{(\cg)}$
and $Z_{N, K,M}^{(\cg')}$, which depend on a further parameter $K\in\N$,
and we will show that
\begin{equation*}
	Z_N^{\beta_N} \approx Z_{N, K,M}^{(\cg)} \,, \qquad
	Z_{N, K,M}^{(\cg)} \approx Z_{N, K,M}^{(\cg')} \,, \qquad
	Z_{N, K,M}^{(\cg')} \approx Z_{N, M}^{(\diff)} \,,
\end{equation*}
where $\approx$ denotes closeness in $L^2$ when we let
$N\to\infty$, then $K\to\infty$ and finally $M\to\infty$.
More precisely, we are going to prove the following relations:
\begin{gather}\label{approx1}
	\limsup_{M\to\infty} \ \limsup_{K\to\infty} \ \limsup_{N\to\infty} \
	\big\| Z_N^{\beta_N} - Z_{N,K,M}^{(\cg)} \big\|_{L^2} = 0 \,, \\
	\label{approxdiff}
	\limsup_{M\to\infty} \ \limsup_{K\to\infty} \ 
	\limsup_{N\to\infty} \
	\big\| Z_{N,K,M}^{(\cg)} - Z_{N,K,M}^{(\cg')} \big\|_{L^2} = 0 \,, \\
	\label{approx2hat}
	\limsup_{M\to\infty} \ \limsup_{K\to\infty} \ \limsup_{N\to\infty} \
	\big\| Z_{N,K,M}^{(\cg')} -  Z_{N,M}^{(\diff)} \big\|_{L^2} = 0 \,,
\end{gather}
which together yield \eqref{eq:second-step}. We accordingly split the proof in three steps.

\subsubsection{Step 1: definition of $Z_{N,K,M}^{(\cg)}$ and
proof of \eqref{approx1}.} 

Let us fix $M, K, N\in\N$ with $1 \ll M \ll K \ll N$.
Our first coarse-graining approximation $Z_{N,K,M}^{(\cg)}$
of the partition function $Z_N^{\beta_N}$ in \eqref{ZconXdom} 
is obtained by \emph{suitably restricting the sums over
$\underline{b}, \underline{b}'$ and $\underline{z}, \underline{z}'$}:
\begin{equation}\label{Zapprox1}
\begin{split}
	Z_{N, K,M}^{(\cg)}
	:= 1+ \sum_{\ell=1}^{\infty} \ & 
	\sum_{\underline{j} \in \{1,\ldots, M\}^\ell_\ll} \ 
	\sum_{(\underline{b}, \underline{b}') \in \cB^\ell(\underline{j})} \
	\sum_{(\underline{z}, \underline{z}') \in \cS^\ell(\underline{b}, \underline{b}')} \
	\prod_{i=1}^\ell X_{N,[b'_{i-1},b_i;b_i']}^{\dom}(z'_{i-1},z_i;z'_i) \, \,,
\end{split}
\end{equation}
where we sum over $\underline{j} = (j_1,\ldots, j_\ell)$ in the following set: 
\begin{equation} \label{eq:restriction-j}
\begin{split}
	\{1,\ldots, M\}^\ell_\ll 
	:= \Big\{ \, 1 \le j_1 < \ldots < j_\ell \le M:
	\quad
	j_i - j_{i-1}\ge 2 \ \ \forall i=2,\ldots, \ell \Big\} \,,
\end{split}
\end{equation}
then, given $\underline{j} = (j_1,\ldots, j_\ell)$, we 
sum over $(\underline{b}, \underline{b}')$ in the set
\begin{equation} \label{eq:restriction-b}
\begin{split}
	\cB^\ell(\underline{j}) 
	:= \Big\{ \, (\underline{b}, \underline{b}') \in \N^\ell\times\N^\ell: \quad
	& b_i \in ( N^{\frac{j_i-1}{M}},  \tfrac{1}{K} N^{\frac{j_i}{M}} ] \,, \
	b_i' \in [b_i, K b_i ] \quad \forall i=1,\ldots, \ell \ \Big\} \,,
\end{split}
\end{equation}
and finally, given $(\underline{b}, \underline{b}')$,
we sum over $\underline{z}, \underline{z}'$ in the ``diffusive set''
\begin{equation*}
\begin{split}
	\cS^\ell(\underline{b}, \underline{b}') 
	:= \Big\{ \, (\underline{z}, \underline{z}') \in (\Z^2)^\ell \times (\Z^2)^\ell:
	\quad & |z_i| \le K \sqrt{b_i} \,, \
	|z'_i| \le K^2 \sqrt{b_i} \quad \forall i=1,\ldots,\ell  \Big\} \,.
\end{split}
\end{equation*}

To see that \emph{$Z_{N, K,M}^{(\cg)}$ in \eqref{Zapprox1} is a restriction of
$Z_N^{\beta_N}$ in \eqref{ZconXdom}},
note that for $(\underline{b}, \underline{b}') \in \cB^\ell(\underline{j})$ we have
$0 < b_1 \le b'_1 < \ldots < b_\ell \le b'_\ell \le N$, and
for large $N$ we also have
$b_i - b'_{i-1} > b_{i-1}$ for $i\ge 2$, because
$b_i > N^{\frac{j_i-1}{M}} \ge N^{\frac{j_{i-1}+1}{M}}
\ge K N^{\frac{1}{M}} b_{i-1}$ (recall that $j_i-j_{i-1}\ge 2$) hence
\begin{equation*}
\begin{split}
	b_i - b'_{i-1} > K N^{\frac{1}{M}} b_{i-1} - K b_{i-1} =
	(N^{\frac{1}{M}}-1) K \, b_{i-1} > b_{i-1} \qquad
	\text{for } N > 2^M \,.
\end{split}
\end{equation*}
Thus the range of the sums in \eqref{Zapprox1} is included in
the range of the sums in \eqref{ZconXdom}.
Since the terms in the polynomial chaos \eqref{eq:Zpoly} are orthogonal in $L^2$, 
it follows that
\begin{equation} \label{eq:invi}
	\big\| Z_N^{\beta_N} - Z_{N,K,M}^{(\cg)} \big\|_{L^2}^2 =
	\big\| Z_N^{\beta_N} \big\|_{L^2}^2 - \big\| Z_{N,K,M}^{(\cg)} \big\|_{L^2}^2 \,,
\end{equation}
hence to prove \eqref{approx1} it suffices to show that
\begin{gather}\label{eq:goal}
	\limsup_{N\to\infty} \ \bbE \big[ \big(Z_N^{\beta_N} \big)^2 \big] \le \frac{1}{1 - \hat{\beta}^2} \,,
	\\
	\label{eq:goal2}
	\liminf_{M\to\infty} \ \liminf_{K\to\infty} \ \liminf_{N\to\infty}  \
	\bbE \big[ \big( Z_{N,K,M}^{(\cg)} \big)^2 \big] \ge \frac{1}{1 - \hat{\beta}^2} \,.
\end{gather}

Relation \eqref{eq:goal} can be easily deduced from the expression \eqref{eq:2mom}.
Indeed, enlarging the sums to $1 \le n_j - n_{j-1} \le N$ and recalling
the definition \eqref{eq:RN} of $R_N$, we get
\begin{equation} \label{eq:argu}
	\begin{split}
	\bbE \big[ \big( Z_N^{\beta_N}  \big)^2 \big] &\le 1+ \sum_{k=1}^\infty (\sigma_N^2)^k 
	\sum_{\substack{1 \le n_j-n_{j-1}\le N\\ j=1,\ldots,k} } \ 
	\sum_{\substack{x_0:=0, \ x_1, \ldots, x_k \in \Z^2 }} \
	\prod_{j=1}^k q_{n_j-n_{j-1}}(x_j-x_{j-1})^2\\
	&= 1+ \sum_{k=1}^\infty (\sigma_N^2)^k
	\, \bigg( \sum_{n=1}^N \sum_{x \in \Z^2} q_n(x)^2 \bigg)^k
	=1+ \sum_{k=1}^\infty \big( \sigma_N^2 R_N\big)^k
	= \frac{1}{1-\sigma_N^2 R_N} \,.
	\end{split}
\end{equation}
Since $\sigma_N \sim \beta_N \sim \hat\beta \sqrt{\pi} / \sqrt{\log N}$,
see \eqref{eq:eta} and \eqref{eq:sub-critical}, and since $R_N \sim \frac{1}{\pi} \log N$,
see \eqref{eq:RN}, we see that \eqref{eq:goal} is proved.

We next prove \eqref{eq:goal2}.
By definition \eqref{Zapprox1} of $Z^{(\cg)}_{N, K,M}$,
in analogy with \eqref{eq:reca}, we have
\begin{equation} \label{eq:estZ10}
\begin{split}
	\bbE\Big[ \big(Z_{N, K,M}^{(\cg)}\big)^2 \Big]
	&= 1+ \sum_{\ell=1}^{\infty} \ 
	\sum_{\underline{j} \in \{1,\ldots, M\}^\ell_\ll} \,
	\sum_{\substack{(\underline{b}, \underline{b}') \in \cB^\ell(\underline{j}) \\
	(\underline{z}, \underline{z}') \in \cS^\ell(\underline{b}, \underline{b}')}} \,
	\prod_{i=1}^\ell \bbE \Big[ \big( X_{N,[b'_{i-1},b_i;b_i']}^{\dom}(z'_{i-1},z_i;z'_i)\big)^2 \Big] \,.
\end{split}
\end{equation}
We now give a lower bound on 
$\bbE \big[ \big( X_{N,[b'_{i-1},b_i;b_i']}^{\dom}(z'_{i-1},z_i;z'_i)\big)^2 \big]$
when we sum over $b_i, b'_i$ and $z_i, z'_i$ 
in the sets $\cB^\ell(\underline{j})$ and $\cS^\ell(\underline{b}, \underline{b}')$.
The next result is proved in Appendix~\ref{sec:lemmaXdomlow}.
	
\begin{lemma}\label{lemmaXdomlow}
For $N,M,K \in \N$ and $j \in \{1, \ldots, M\}$, define
\begin{equation} \label{eq:Xi} 
	\Xi_{N, M,K}(j) := 
	\inf_{\substack{0 \le a \le N^{\frac{(j-2)^+}{M}} \\
	\rule{0pt}{.8em}|x| \le K^2 \sqrt{a}}} \
	\sum_{\substack{b \in ( N^{\frac{j-1}{M}}, \frac{1}{K} N^{\frac{j}{M}} ] \\ 
	b' \in \left[b, K b \right]}} \
	\sum_{\substack{|z| \le K \sqrt{b} \\
	|z'| \le K^2 \sqrt{b}}} \
	 \bbE \Big[ \big( X^{\rm{dom}}_{N,[a,b_;b']}(x,z;z') 
	\big)^2 \Big] \,.
\end{equation}
Then, for any $M \in \N$ and $j \in \{1,\ldots, M\}$, we have
\begin{equation} \label{eq:Ij}
	\liminf_{K \to\infty} \ \liminf_{N\to\infty} \ \Xi_{N, M,K}(j)
	\,=\, I_M(j) := \int_{\frac{j-1}{M}}^{\frac{j}{M}} 
	\frac{\hat\beta^2}{1-\hat{\beta}^2 s } \dd s \,.
\end{equation}
\end{lemma}

Coming back to \eqref{eq:estZ10},
by definition \eqref{eq:Xi} of $\Xi_{N, M,K}(j)$, we have the lower bound
\begin{equation} \label{eq:estZ1}
\begin{split}
	\bbE\Big[ \big(Z_{N, K,M}^{(\cg)}\big)^2 \Big]
	& \,\ge\, 1 + \sum_{\ell=1}^{\infty} \ 
	\sum_{\underline{j} \in \{1,\ldots, M\}^\ell_\ll} \ 
	\prod_{i=1}^\ell \Xi_{N, M,K}(j_i) \,,
\end{split}
\end{equation}
which yields, by \eqref{eq:Ij},
\begin{equation} \label{eq:quasiliminf}
	\liminf_{K\to\infty} \ \liminf_{N\to\infty} \ 
	\bbE\Big[ \big(Z_{N, K,M}^{(\cg)}\big)^2 \Big] \,\ge\, 
	1 + \sum_{\ell=1}^{\infty} \ 
	\sum_{\underline{j} \in \{1,\ldots, M\}^\ell_\ll} \ 
	\prod_{i=1}^\ell I_M(j_i) \,.
\end{equation}
Recalling the definition \eqref{eq:restriction-j} of $\{1,\ldots, M\}^\ell_\ll$,
we can rewrite the r.h.s.\ of \eqref{eq:quasiliminf} as
\begin{equation*}
	1 + \sum_{\ell=1}^{\infty} \,
	\frac{1}{\ell !} \, \Bigg\{ \bigg(\sum_{j=1}^M I_M(j) \bigg)^\ell  \,-\, 
	\sum_{\substack{j_1,\ldots,j_\ell  \in \{1,\ldots, M\}\\ 
	\exists h  \ne k: \ |j_h - j_{k}| \le 1}} I_M(j_1) \,\cdots\, I_M(j_\ell) \Bigg\} \,.
\end{equation*}
The second term gives a vanishing contribution as $M\to\infty$, because
$\max_{1 \le j \le M} I_M(j) \le \frac{C}{M}$,
with $C := \frac{\hat{\beta}^2}{1 - \hat{\beta}^2} < \infty$, hence
\begin{equation*}
\begin{split}
	\sum_{\ell=1}^{\infty} \,
	\frac{1}{\ell !} \, \sum_{\substack{j_1,\ldots,j_\ell  \in \{1,\ldots, M\}\\ 
	\exists h  \ne k: \ |j_h - j_{k}| \le 1}} I_M(j_1) \,\cdots\, I_M(j_\ell)
	& \le 
	\sum_{\ell=1}^{\infty} \,
	\frac{1}{\ell !} \, 
	\frac{C^\ell}{M^\ell} \,  \binom{\ell}{2} 3 M^{\ell - 1}
	\,=\, \frac{C'}{M}
	\,\xrightarrow{\,M\to\infty\,}\, 0 \,,
	\end{split}
\end{equation*}
where $\binom{\ell}{2}$ is the number of pairs
$\{h,k\}$ with $h \ne k$ and $3M^{\ell -1}$ bounds the number of choices of 
$j_1, \dots, j_\ell $ with $j_h \in \{j_k -1, \, j_k, \, j_k +1\}$.
Since $\sum_{j=1}^M I_M(j) = \int_0^1 \frac{\hat\beta^2}{1-\hat{\beta}^2 s } \dd s 
= \log \frac{1}{1-\hat\beta^2}$, we have finally shown that
\begin{equation} \label{eq:lastZ1}
	\liminf_{M\to\infty} \ \liminf_{K\to\infty} \ \liminf_{N\to\infty} \ 
	\bbE\Big[ \big(Z_{N, K,M}^{(\cg)}\big)^2 \Big] \,\ge\, 
	1 + \sum_{\ell=1}^{\infty} \, \frac{1}{\ell !} \,
	\Big( \log\tfrac{1}{1-\hat\beta^2}\Big)^\ell \,=\, \frac{1}{1-\hat\beta^2} \,,
\end{equation}
which is \eqref{eq:goal2}.
This completes the proof of \eqref{approx1}.\qed

\subsubsection{Step 2: definition of $Z_{N,K,M}^{(\cg')}$ and
proof of \eqref{approxdiff}.}

Starting from $Z_{N,K,M}^{(\cg)}$ in \eqref{Zapprox1}, we
set $b'_{i-1} = 0$ and $z'_{i-1}=0$ inside each $X_{N}^{\dom}$
to obtain our second approximation: 
\begin{equation}\label{Zapprox2}
\begin{split}
	Z_{N, K,M}^{(\cg')}
	:= 1+ \sum_{\ell=1}^{\infty} \ & 
	\sum_{\underline{j} \in \{1,\ldots, M\}^\ell_\ll} \ 
	\sum_{(\underline{b}, \underline{b}') \in \cB^\ell(\underline{j})} \
	\sum_{(\underline{z}, \underline{z}') \in \cS^\ell(\underline{b}, \underline{b}')} \
	\prod_{i=1}^\ell X_{N,[0,b_i;b_i']}^{\dom}(0,z_i;z'_i) \,.
\end{split}
\end{equation}
Heuristically, the reason why we set $b'_{i-1} = 0$ is that 
$b_i \gg b'_{i-1}$, hence $b_i-b_{i-1}' \approx b_i$ (indeed, note that
$b_i \ge N^{\frac{j_i-1}{M}} \gg N^{\frac{j_{i-1}}{M}} \ge b'_{i-1}$
since $j_i - 1 > j_{i-1}$,
see \eqref{eq:restriction-b} and \eqref{eq:restriction-j}).

\smallskip

We need to prove \eqref{approxdiff}. 
Given $\underline{b}, \underline{b}'$
and $\underline{z}, \underline{z}'$, let us introduce the shortcuts
\begin{equation} \label{eq:XY}
	X_i := X_{N,[b'_{i-1},b_i;b_i']}^{\dom}(z'_{i-1},z_i;z'_i) \,, \qquad
	Y_i :=  X_{N,[0,b_i;b_i']}^{\dom}(0,z_i;z'_i) \,,
\end{equation}
so that, comparing \eqref{Zapprox1} and \eqref{Zapprox2}, we can write 
\begin{equation*}
\begin{split}
	Z_{N,K,M}^{(\cg')} &- Z_{N,K,M}^{(\cg)} 
	= \sum_{\ell=1}^{\infty} \ 
	\sum_{\underline{j} \in \{1,\ldots, M\}^\ell_\ll} \ 
	\sum_{(\underline{b}, \underline{b}') \in \cB^\ell(\underline{j})} \
	\sum_{(\underline{z}, \underline{z}') \in \cS^\ell(\underline{b}, \underline{b}')}
	\Bigg( \prod_{i=1}^\ell Y_i - \prod_{i=1}^\ell X_i \Bigg) \\
	&= \sum_{\ell=1}^{\infty} \
	\sum_{\underline{j} \in \{1,\ldots, M\}^\ell_\ll} \ 
	\sum_{(\underline{b}, \underline{b}') \in \cB^\ell(\underline{j})} \
	\sum_{(\underline{z}, \underline{z}') \in \cS^\ell(\underline{b}, \underline{b}')}
	\, \sum_{h=1}^\ell \,
	\bigg\{ \prod_{i=1}^{h-1} Y_i \bigg\} \, (Y_h - X_h) \,
	\bigg\{ \prod_{i=h+1}^\ell  X_i \bigg\} \,,
\end{split}
\end{equation*}
and note that different terms in the sums are orthogonal in $L^2$.
We justify below the following key estimate, see Lemma~\ref{th:clai}: 
for any $\varepsilon>0$, for $N$ large enough, we can bound for all $i=1,\ldots, \ell$
\begin{equation} \label{eq:clai}
	\bbE\big[ (Y_i - X_i)^2 \big] \le \epsilon^2 \, \bbE[Y_i^2] \,.
\end{equation}
By the triangle inequality, this implies $\bbE[X_i^2]^{1/2} \le (1+\epsilon) \bbE[Y_i^2]^{1/2}
\le 2 \, \bbE[Y_i^2]^{1/2}$, hence
\begin{equation*}
\begin{split}
	\bbE \big[ \big(Z_{N,K,M}^{(\cg')} - Z_{N,K,M}^{(\cg)} \big)^2 \big] 
	 & \le \sum_{\ell=1}^{\infty} \,
	\sum_{\underline{j} \in \{1,\ldots, M\}^\ell_\ll} \,
	\sum_{(\underline{b}, \underline{b}') \in \cB^\ell(\underline{j})} \,
	\sum_{(\underline{z}, \underline{z}') \in \cS^\ell(\underline{b}, \underline{b}')}
	\bigg( \epsilon^2 \sum_{h=1}^\ell 2^{2(\ell-h)} \bigg) 
	\prod_{i=1}^{\ell} \bbE[Y_i^2] \\
	&  \le \epsilon^2 \,\sum_{\ell=1}^{\infty} \,
	4^\ell \sum_{\underline{j} \in \{1,\ldots, M\}^\ell_\ll} \ 
	\sum_{(\underline{b}, \underline{b}') \in \cB^\ell(\underline{j})} \
	\sum_{(\underline{z}, \underline{z}') \in \cS^\ell(\underline{b}, \underline{b}')}
	\, \prod_{i=1}^{\ell} \bbE[Y_i^2]  \,,
\end{split}
\end{equation*}
because $\sum_{h=1}^\ell 2^{2(\ell-h)}
= \frac{4^{\ell} - 1}{4-1}
\le 4^\ell$.
We now enlarge the sum ranges to obtain the factorization
\begin{equation} \label{eq:plugi}
\begin{split}
	& \bbE \big[ \big(Z_{N,K,M}^{(\cg')} - Z_{N,K,M}^{(\cg)} \big)^2 \big] \\
	& \qquad \le \epsilon^2 \, \sum_{\ell=1}^{\infty} \,
	4^\ell \sum_{1 \le j_1 < j_2 < \ldots < j_\ell \le M} \ 
	\, \prod_{i=1}^{\ell} \Bigg\{\,
	\sum_{b_i \le b_i' \in ( N^{\frac{j_i-1}{M}}, N^{\frac{j_i}{M}}]} \
	\sum_{z_i, z'_i \in \Z^2} \ \bbE[Y_i^2] \Bigg\} \,.
\end{split}
\end{equation}
The following asymptotics on the term in brackets is proved in
Appendix~\ref{sec:brackets}.

\begin{lemma}\label{th:brackets}
For any $M \in \N$ and $j\in\{1,\ldots, M\}$ we have
\begin{equation} \label{eq:thebound}
\begin{split}
	\lim_{N\to\infty}  \ \Bigg\{ \,
	\sum_{\substack{b \le b' \in ( N^{\frac{j-1}{M}}, N^{\frac{j}{M}}] \\ z,z'\in\Z^2}} \
	\bbE\big[  X_{N,[0,b;b']}^{\dom}(0,z;z')^2 \big] \,\Bigg\} & \,=\, 
	I_M(j) \,=\, \int_{\frac{j-1}{M}}^{\frac{j}{M}} 
	\frac{\hat\beta^2}{1-\hat{\beta}^2 s } \, \dd s \,.
\end{split}
\end{equation}
\end{lemma}

We can plug \eqref{eq:thebound} into \eqref{eq:plugi} 
(where the sum is finite: it can be stopped at $\ell = M$, since for $\ell > M$
there is no choice of $1 \le j_1 < j_2 < \ldots < j_\ell \le M$), which yields
\begin{equation} \label{eq:inan}
\begin{split}
	\limsup_{N\to\infty} \ \bbE \big[ \big(Z_{N,K,M}^{(\cg')} - Z_{N,K,M}^{(\cg)} \big)^2 \big] 
	\le \epsilon^2 \, \sum_{\ell=1}^{\infty} \,
	4^\ell
	\sum_{1 \le j_1 < j_2 < \ldots < j_\ell \le M} \ 
	\, \prod_{i=1}^{\ell} \, I_M(j_i) \qquad\qquad &  \\
	\le \epsilon^2 \, \sum_{\ell=1}^{\infty} \,
	\frac{4^\ell}{\ell!}
	\bigg(\sum_{j=1}^{M} I_M(j) \bigg)^\ell  
	\le \epsilon^2 \, \exp\bigg(
	4 \sum_{j=1}^{M} I_M(j) \bigg)
	= \frac{\epsilon^2}{(1-\hat\beta^2 )^{4}} \,. &
\end{split}
\end{equation}
This completes the proof of \eqref{approxdiff}, since we can take $\epsilon > 0$
as small as we wish.

\smallskip

It only remains to justify \eqref{eq:clai}. The following result is proved in Appendix~\ref{sec:clai}.

\begin{lemma}\label{th:clai}
Given $K,M\in\N$ and $\epsilon > 0$, there exists
$N_0 = N_0(\epsilon, M, K) < \infty$ such that for all $N > N_0$ the following bound holds:
\begin{equation}\label{eq:estXhatest}
\begin{gathered}
	\bbE \big[\big(X_{N,[a,b;b']}^{\dom}(x,z;z')- X_{N,[0,b;b']}^{\dom}
	(0,z;z') \big)^2\big] \le \varepsilon^2 \, 
	\bbE \big[X_{N,[0,b;b']}^{\dom}(0,z;z') ^2\big] \,,
\end{gathered} 
\end{equation}
uniformly for $(a,x), (b,z), (b',z') \in \Z^3_{\mathrm{even}}
= \{y \in \Z^3: \ y_1 + y_2 + y_3 \text{ is even}\}$  such that,
for some $j \in \{1,\ldots, M\}$,
\begin{equation}\label{eq:estXhat}
\begin{gathered}
	a \in [0, N^{\frac{(j-2)^+}{M}}] \,, \quad
	b \in (N^{\frac{j-1}{M}}, N^{\frac{j}{M}}] \,,  \quad
	|x| \le K^2 \sqrt{a} \,, \quad
	|z| \le K \sqrt{b} \,.
\end{gathered} 
\end{equation}
\end{lemma}

\subsubsection{Step 3: proof of \eqref{approx2hat}}

Recalling \eqref{Xdomj},
we can rewrite $Z_{N, M}^{(\diff)}$ in \eqref{Z2multform} as follows:
\begin{equation}\label{Zapprox3}
\begin{split}
	Z_{N,M}^{(\diff)} = 1+ \sum_{\ell=1}^{\infty} \,
	\sum_{1 \le j_1 < j_2 < \ldots < j_\ell \le M} 
	\sum_{\substack{\underline{b}, \underline{b}' \in \N^\ell: \\ 
	b_i \le b_i' \in ( N^{\frac{j_i-1}{M}}, N^{\frac{j_i}{M}} ]}} 
	\sum_{\underline{z}, \underline{z}' (\Z^2)^\ell} \,
	\prod_{i=1}^\ell X_{N,[0,b_i;b_i']}^{\dom}(0,z_i;z'_i) \,.
\end{split}
\end{equation}
By \eqref{Zapprox2}, we see that $Z_{N,K,M}^{(\cg')}$
is a \emph{restriction} of the sum which defines $Z_{N,M}^{(\diff)}$,
therefore
\begin{equation*}
	\big\| Z_{N,K,M}^{(\cg')} -  Z_{N,M}^{(\diff)} \big\|_{L^2}^2
	= \big\| Z_{N,M}^{(\diff)} \big\|_{L^2}^2
	- \big\| Z_{N,K,M}^{(\cg')} \big\|_{L^2}^2 \,.
\end{equation*}
Then, to prove \eqref{approx2hat}, it is enough to show that
\begin{gather}\label{eq:1ste}
	\liminf_{M\to\infty} \ \liminf_{K\to\infty} \ \liminf_{N\to\infty} \
	\bbE\big[ \big(Z_{N,K,M}^{(\cg')}\big)^2 \big] \,\ge\, \frac{1}{1-\hat\beta^2} \,, \\
	\label{eq:2ste}
	\forall M\in\N: \qquad \limsup_{N\to\infty} \
	\bbE\big[ \big( Z_{N,M}^{(\diff)}\big)^2 \big] \,\le\, \frac{1}{1-\hat\beta^2} \,.
\end{gather}

We first consider \eqref{eq:1ste}.
Recalling \eqref{Zapprox2}, in analogy with \eqref{eq:reca}, we can write
\begin{equation*}
	\bbE\big[ \big(Z_{N,K,M}^{(\cg')}\big)^2 \big]
	= 1+ \sum_{\ell=1}^{\infty} \ 
	\sum_{\underline{j} \in \{1,\ldots, M\}^\ell_\ll} \ 
	\sum_{(\underline{b}, \underline{b}') \in \cB^\ell(\underline{j})} \
	\sum_{(\underline{z}, \underline{z}') \in \cS^\ell(\underline{b}, \underline{b}')} \
	\prod_{i=1}^\ell
	\bbE \big[ X_{N,[0,b_i;b_i']}^{\dom}(0,z_i;z'_i)^2 \big] \,.
\end{equation*}
We can now use the quantity $\Xi_{N,M,K}(j_i)$ defined in \eqref{eq:Xi} to bound
\begin{equation*}
\begin{split}
	\bbE\big[ \big(Z_{N,K,M}^{(\cg')}\big)^2 \big]
	& \,\ge\, 1 + \sum_{\ell=1}^{\infty} \ 
	\sum_{\underline{j} \in \{1,\ldots, M\}^\ell_\ll} \ 
	\prod_{i=1}^\ell  \, \Xi_{N, M,K}(j_i) \,,
\end{split}
\end{equation*}
which coincides with the r.h.s.\ of \eqref{eq:estZ1}. As a consequence,
the bounds from \eqref{eq:quasiliminf} to \eqref{eq:lastZ1}
apply verbatim to $\bbE\big[ \big(Z_{N,K,M}^{(\cg')}\big)^2 \big]$
and show that \eqref{eq:1ste} holds.

We finally consider \eqref{eq:2ste}, which we have essentially already proved. Indeed, note that
$\bbE\big[ \big(Z_{N,K,M}^{(\diff)}\big)^2 \big]$
is given by the second line of \eqref{eq:plugi}
where we replace $\epsilon^2$ and $4^\ell$ by~$1$.
When we apply the limit \eqref{eq:thebound}, we obtain an analogue of \eqref{eq:inan},
again with $\epsilon^2$ and $4^\ell$ replaced by~$1$, which yields
precisely \eqref{eq:2ste}.
This completes the proof of Lemma~\ref{th:brackets}.\qed

\subsection{Proof of Lemma~\ref{th:third-step}}

\label{sec:proof-third-step}

We recall that the event $A_{N,M}$ was defined in 
\eqref{eq:ANM}.
In order to prove~\eqref{approxlogXdom}, it is enough to show that
the following three relations hold:
\begin{gather}
	\label{approxlogZ-1}
	\lim_{M \rightarrow \infty} \ \limsup_{N \to\infty} \ \bbP \bigg( 
	\bigg| \log Z_{N,M}^{(\diff)}- 
	\sum_{j=1}^M \big\{ X_{N,M}^{\dom}(j) -\tfrac{1}{2} X_{N,M}^{\dom}(j)^2 \big\} 
	\bigg| >\varepsilon \,, \ A_{N,M} \bigg) = 0 \,, \\
	\label{approxXdom-1}
	\lim_{M\to\infty} \ \limsup_{N \rightarrow \infty} \ \bigg\| 
	\sum_{j=1}^M X_{N,M}^\dom(j) - X_N^{\dom} \bigg\|_{L^2}=0 \,, \\
	\label{approxXdom-2}
	\lim_{M\to\infty} \ \limsup_{N \rightarrow \infty} \ \bigg\| 
	\sum_{j=1}^M X_{N,M}^\dom(j) ^2
	- \bbE[(X_N^\dom)^2]  \bigg\|_{L^1}=0 \,.
\end{gather}
We are going to exploit the following result.

\begin{lemma}\label{lem:hyper}
Fix $\hat\beta < 1$. For every $M\in\N$ and $j\in\{1,\ldots, M\}$ we have
\begin{equation}\label{M1}
	\lim_{N\to\infty}
	\ \bbE \big[ X_{N,M}^{\dom}(j)^2\big]
	\,=\, \int_{\frac{j-1}{M}}^{\frac{j}{M}} \frac{\hat{\beta}^2}{1-\hat{\beta}^2s} 
	\, \dd s
	\,\le\, \frac{c}{M} \,, \qquad \text{with} \quad
	c = c_{\hat\beta} := \tfrac{\hat\beta^2}{1-\hat\beta^2} \,.
\end{equation}
Moreover, there exist $p_{\hat{\beta}}>2$ 
and $C=C_{\hat{\beta}}<\infty$ such that for all $2 < p \le p_{\hat\beta}$
\begin{equation}\label{M2}
	\forall M\in\N \,, \ \forall j\in \{1,\ldots,M\} : \qquad
	\limsup_{N\to\infty} \ \bbE
	\ \big[|X_{N,M}^{\dom}(j)|^p\big] \le \frac{C}{M^{\frac{p}{2}}} \,.
\end{equation}
\end{lemma}

\begin{proof}
Relation \eqref{M1} is already proved in \eqref{eq:thebound},
by the definition \eqref{Xdomj} of $X_{N,M}^{\dom}(j)$.

Intuitively, the bound \eqref{M2} holds because $\bbE \big[ |X_{N,M}^{\dom}(j)|^p\big]
\le C \, \bbE \big[ X_{N,M}^{\dom}(j)^2\big]^{\frac{p}{2}}$
by the \emph{hypercontractivity of polynomial  chaos}.
The details are presented in Appendix~\ref{sec:M2}.
\end{proof}

\smallskip

It only remains to prove \eqref{PA_NM} and the three relations
\eqref{approxlogZ-1}-\eqref{approxXdom-2}.

\medskip
\noindent
\emph{Proof of \eqref{PA_NM}.}
For any $p > 2$ we can bound, by Markov's inequality,
\begin{equation*}
	\bbP\big( (A_{N,M})^c \big) \le \sum_{j=1}^{M} \,
	\bbP \big(|X_{N,M}^{\dom}(j)|>\tfrac{1}{2}\big) 
	\le M \, 2^p \, \max_{j\in\{1,\ldots, M\}} \bbE \big[ |X_{N,M}^{\dom}(j)|^p \big] \,, 
\end{equation*}
and relation \eqref{PA_NM} follows directly by \eqref{M2}.\qed

\medskip
\noindent
\emph{Proof of \eqref{approxlogZ-1}.}
By \eqref{Z2multform} we can write 
$\log Z_{N,M}^{(\diff)} = \sum_{j=1}^M \log (1+X_{N,M}^{\dom}(j))$.
If we fix $2<p<\min \{3, \, p_{\hat{\beta}}\}$, with $p_{\hat{\beta}}$ 
as in Lemma~\ref{lem:hyper},
we can bound $| \log (1+x) - \{x-\frac{1}{2}x^2\} | \le c |x|^p$
for $|x| \le \frac{1}{2}$, hence
\begin{equation*}
	\bbE \Bigg[ \bigg| \log Z_{N,M}^{(\diff)}- 
	\sum_{j=1}^M \big\{ X_{N,M}^{\dom}(j) -\tfrac{1}{2} X_{N,M}^{\dom}(j)^2 \big\} 
	\bigg| \, \ind_{A_{N,M}} \Bigg]
	\,\le\, c \, \sum_{j=1}^M \bbE \big[ |X_{N,M}^{\dom}(j)|^p \big]
	\,\le\, c \, \frac{C}{M^{\frac{p}{2}-1}} \,,
\end{equation*}
which proves \eqref{approxlogZ-1}, by Markov's inequality.
\qed

\medskip
\noindent
\emph{Proof of \eqref{approxXdom-1}.}
The polynomial chaos $\sum_{j=1}^M X_{N,M}^{\dom}(j)$ contains less terms than $X_N^\dom$, 
therefore to prove \eqref{approxXdom-1} it is enough to show that for any fixed $M\in\N$
\begin{equation} \label{eq:wehava0}
	\lim_{N\to\infty} \, \bbE \Bigg[\bigg(\sum_{j=1}^M X_{N,M}^{\dom}(j)\bigg)^2\Bigg]
	=\lim_{N \to \infty} \, \bbE \big[\big(X_N^\dom \big)^2\big] = 
	\int_0^1 \frac{\hat{\beta}^2}{1-\hat{\beta}^2s} \, \text{d}s
\end{equation}
where the last equality follows by \eqref{M1}, because $X_N^\dom$ equals $X_{N,M}^\dom(j)$ 
for $M=j=1$ (cf.\ \eqref{eq:XNdom} with \eqref{Xdomj} and \eqref{XNdom}).
Since the variables $X_{N,M}^\dom(j)$'s are centered and independent, a further application 
of \eqref{M1} yields
\begin{equation} \label{eq:wehava}
	\bbE \Bigg[\bigg(\sum_{j=1}^M X_{N,M}^{\dom}(j)\bigg)^2\Bigg]
	=\sum_{j=1}^M \bbE \big[ X_{N,M}^\dom(j)^2 \big] 
	\, \xrightarrow[]{\,N\to\infty\,}\,
	\sum_{j=1}^M \, I_M(j)
	= \int_0^1 \frac{\hat{\beta}^2}{1-\hat{\beta}^2s} 
	\, \text{d}s \,,
\end{equation}
as desired. This completes the proof.
\qed

\medskip
\noindent
\emph{Proof of \eqref{approxXdom-2}.}
In view of the first equalities in \eqref{eq:wehava0} and \eqref{eq:wehava},
it suffices to show that
\begin{equation} \label{eq:wlln}
	\lim_{M\to\infty} \ \limsup_{N\to\infty} \
	\bigg\| \sum_{j=1}^M \big\{ X_{N,M}^\dom(j)^2
	- \bbE\big[ X_{N,M}^\dom(j)^2 \big] \big\}
	\bigg\|_{L^1} = 0 \,.
\end{equation}
This is a weak law of large numbers
for the independent random variables $W_{j} := X_{N,M}^\dom(j)^2$,
which satisfy the following Lyapunov condition (by  \eqref{M2} with $q := p/2$):
\begin{equation} \label{eq:Lya}
	\exists q = q_{\hat\beta} > 1, \ C = C_{\hat\beta} < \infty: \qquad
	\forall M \in \N\qquad
	\limsup_{N\to\infty} \ \max_{j\in\{1,\ldots, M\}} \
	\bbE[W_{j}^q] \le \frac{C}{M^{q}} \,.
\end{equation}

We prove \eqref{eq:wlln} by truncation
at level $T_M := M^{-\alpha}$, for an arbitrary $\alpha \in (\frac{1}{2},1)$. Note that
\begin{equation*}
\begin{split}
	\bigg\| \sum_{j=1}^M W_{j} \, \ind_{\{W_{j} > T_M\}}
	\bigg\|_{L^1}
	& = \sum_{j=1}^M \bbE\big[ W_{j} \, \ind_{\{W_{j} > T_M\}}\big]
	\le \sum_{j=1}^M \frac{\bbE[W_{j}^q]}{T_M^{q-1}}
	\le M^{1+\alpha(q - 1)} \, \max_{j\in\{1,\ldots, M\}} \bbE[W_{j}^q] \,,
\end{split}
\end{equation*}
which, by \eqref{eq:Lya}, vanishes as $N\to\infty$ followed by $M\to\infty$
provided $1+\alpha(q-1)-q < 0$, that is $\alpha < 1$. 
To prove  \eqref{eq:wlln} it only remains to show that
\begin{equation*}
	\lim_{M\to\infty} \ \limsup_{N\to\infty} \
	\bigg\| \sum_{j=1}^M \Big\{ W_{j} \, \ind_{\{W_{j} \le T_M\}}
	- \bbE\big[ W_{j} \, \ind_{\{W_{j} \le T_M\}} \big] \Big\}
	\bigg\|_{L^1} = 0 \,.
\end{equation*}
It is simpler to prove convergence in $L^2$, because this follows by a variance computation:
\begin{equation*}
	\bbvar \bigg(\sum_{j=1}^M W_{j} \, \ind_{\{W_{j} \le T_M\}}
	\bigg) =
	\sum_{j=1}^M
	\bbvar \big( W_{j} \, \ind_{\{W_{j} \le T_M\}} \big)
	\le M \, T_M^2
	= M^{1-2\alpha} \,,
\end{equation*}
which vanishes as $M\to\infty$ provided $1-2\alpha<0$, that is
$\alpha > \frac{1}{2}$.
\qed

\subsection{Proof of Lemma~\ref{th:fourth-step}}
\label{sec:proof-fourth-step}

We first prove \eqref{eq:goal1}.
In view of \eqref{approxlogXdom} and \eqref{PA_NM},
it suffices to show that
\begin{equation}\label{eq:goal1bis}
	\forall \epsilon > 0: \qquad
	\lim_{N\to\infty} \, \bbP \big( \big| \log Z_N^{\beta_N} - \log Z_{N,M}^{(\diff)}
	\big| > \epsilon \,,\, A_{N,M} \big) = 0 \,,
\end{equation}
where we recall that the event $A_{N,M} \subseteq \{Z_{N,M}^{(\diff)} > 0\}$
was defined in \eqref{eq:ANM}.

For any $a,b\in \R$ and $\epsilon, \eta \in (0,1)$ we have the
following inclusion:
\begin{equation*}
	\left\lbrace  |\log a - \log b | > \varepsilon \right\rbrace \subseteq 
	\left\lbrace b <  2 \eta \varepsilon\right\rbrace \cup \{ |a-b|> \eta\epsilon^2 \} \,.
\end{equation*}
Indeed, if both $b \ge 2\eta \epsilon$ and $|a-b| \le \eta\epsilon^2$,
then $a \ge b-\eta\epsilon^2 \ge 2\eta\epsilon - \eta\epsilon^2 \ge \eta\epsilon$,
so that both $a,b \in [ \eta \epsilon, \infty)$, hence
$|\log a - \log b| = |\int_a^b \frac{1}{x} \dd x|
\le \frac{1}{\eta\epsilon}|b-a| \le \frac{1}{\eta\epsilon} \eta\epsilon^2
= \epsilon$. It follows that
\begin{equation*}
	\bbP \big( \big| \log Z_N^{\beta_N} - \log Z_{N,M}^{(\diff)}
	\big| > \epsilon ,\, A_{N,M} \big) \le
	\bbP\big( Z_{N,M}^{(\diff)} < 2\eta\epsilon,\, A_{N,M} \big)
	+ \bbP \big(\big| Z_{N}^{\beta_N} - Z_{N,M}^{(\diff)}\big|
	> \eta\epsilon^2 \big)
\end{equation*}
and note that the second term in the r.h.s.\ vanishes as $N \to\infty$
followed by $M\to\infty$, for any fixed $\epsilon, \eta \in (0,1)$, thanks to \eqref{eq:second-step}. 
It remains to show that
\begin{equation*}
	\forall \epsilon > 0: \qquad
	\lim_{\eta \downarrow 0} \
	\limsup_{M\to\infty} \ \limsup_{N\to\infty} \
	\bbP\big( Z_{N,M}^{(\diff)} < 2\eta\epsilon,\, A_{N,M} \big) \,=\, 0 \,.
\end{equation*}
To this purpose, we can bound
\begin{equation*}
\begin{split}
	\bbP\big( Z_{N,M}^{(\diff)} < 2\eta\epsilon,\, A_{N,M} \big)
	&\le \bbP \Big( \big| \log Z_{N,M}^{(\diff)} - 
	\big\{ X_N^\dom - \tfrac{1}{2} \bbE[(X_N^\dom)^2] \big\} 
	\big| > 1 , \, A_{N,M} \Big) \\
	&\qquad + 
	\bbP\Big( X_N^\dom - \tfrac{1}{2} \bbE[(X_N^\dom)^2] < \log(2\eta\epsilon) + 1 \Big)
\end{split}
\end{equation*}
and note that the first term in the r.h.s.\ vanishes as $N \to\infty$
followed by $M\to\infty$, by \eqref{approxlogXdom}. To show that the second term
vanishes as $N\to\infty$ followed by $\eta \downarrow 0$, we fix
$\eta > 0$ small, so that $\log(2\eta\epsilon) + 1 < 0$,
and we apply Markov's inequality to bound, for some $C<\infty$,
\begin{equation*}
	\bbP\Big( X_N^\dom - \tfrac{1}{2} \bbE[(X_N^\dom)^2] < \log(2\eta\epsilon) + 1 \Big)
	\le \frac{\bbE\big[ \big(X_N^\dom - \tfrac{1}{2} \bbE[(X_N^\dom)^2]\big)^2\big]}
	{|\log(2\eta\epsilon) + 1|^2} 
	\le \frac{C}{ |\log(2\eta\epsilon) + 1|^2 } \,,
\end{equation*}
because $\bbE\big[ \big(X_N^\dom - \tfrac{1}{2} \bbE[(X_N^\dom)^2]\big)^2\big]$
converges to a finite limit as $N\to\infty$, see \eqref{eq:wehava0}.

\smallskip

It only remains to prove \eqref{eq:boundLp}.
The second bound in \eqref{eq:boundLp} follows by
\eqref{M2}, because we already remarked that $X_N^\dom = X_{N,M}^\dom(j)$ 
with $j=M=1$, see \eqref{eq:XNdom} and \eqref{Xdomj}, \eqref{XNdom}.
The first bound in \eqref{eq:boundLp} was proved in \cite{CSZ20}
(see equations (3.12), (3.14) and the lines following (3.16))
exploiting \emph{concentration of measure for the left tail of $\log Z_N$}.\qed

\section{Proof of Theorem~\ref{th:XNGauss}}
\label{sec:XNGauss}

We have already noticed in \eqref{eq:wehava0} that
\begin{equation}\label{eq:convvar}
	\lim_{N\to\infty} \bbE \big[(X_N^{\dom})^2\big] = \sigma^2 
	:= \log \tfrac{1}{1-\hat{\beta}^2} \,,
\end{equation}
which follows by \eqref{M1}, because $X_N^\dom = X_{N,1}^\dom(1)$ 
(see \eqref{eq:XNdom} and \eqref{Xdomj}, \eqref{XNdom}).
Therefore we only need to prove that
\begin{equation} \label{convXdom}
	X_N^{\dom} \xrightarrow[]{\ d \ } \cN\big( 0, \sigma^2 \big) \,.
\end{equation}

We can apply Theorem~\ref{th:main} to the polynomial chaos
$X_N^{\dom}$ defined in \eqref{eq:XNdom}. As in the proof
of Theorem~\ref{th:singular}, we can cast $X_N^{\dom}$ in the
form \eqref{eq:polychaosgen-k} with $\bbT := \N \times \Z^2$
and $\eta^N_t = \eta_N(m,z)$ defined in \eqref{eq:eta},
while for $A:= \{t_1,\ldots,t_{k}\}=\{(n_1,x_1),\ldots,(n_{k}, x_{k})\} \subseteq \bbT$
we set
\begin{equation*}
	\begin{split}
		q_N(A) =(\sigma_N)^{k} \, 
		\ind_{ \left\{ \substack{0=:n_0<n_1<\ldots<n_k \le N \\ \max\{n_2-n_1,\ldots,n_k-n_{k-1}\} 
		\le n_1-n_0} \right\} } \, \prod_{j=1}^{k}q_{n_j-n_{j-1}}(x_j-x_{j-1}) \,.
	\end{split}
\end{equation*}
By Theorem~\ref{th:main}, to prove \eqref{convXdom} we need to verify the following conditions:
\begin{enumerate}
\item \emph{Limiting second moment:}
we already showed that
$\lim_{N\to\infty} \bbE [(X_N^\dom)^2]= \sigma^2$, see \eqref{eq:convvar}.

\item \emph{Subcriticality}:
we need to show that
\begin{equation} \label{eq:fromwhich}
		\lim_{K\to \infty} \ \limsup_{N \rightarrow \infty}  \
		\sum_{\substack{A \subset \mathbb{T}\\ |A| \ge K}}q_N(A)^2=0 \,.
\end{equation}
Arguing as in \eqref{eq:argu}, we can enlarge the 
sums to $1\le n_j-n_{j-1}\le N$ and remove the constraint $\max\{n_2-n_1,\ldots,n_k-n_{k-1}\}
\le n_1-n_0$, to get the bound
\begin{equation*}
\begin{split}
	\sum_{\substack{A \subset \mathbb{T}\\ |A| \ge K}}q_N(A)^2
	& \le \sum_{k=K}^\infty (\sigma_N^2)^k 
	\sum_{\substack{1 \le n_j-n_{j-1} \le N\\ j=1,\ldots,k}} 
	\sum_{\substack{x_1,\ldots,x_k \in \Z^2\\ x_0:= 0}} \, \prod_{j=1}^k q_{n_j-n_{j-1}}(x_j-x_{j-1})^2 \\
	&=\sum_{k=K}^\infty (\sigma_N^2)^k \Big( \sum_{n=1}^N \sum_{x \in \Z^2} q_n(x)^2\Big)^{k} 
	= \sum_{k=K}^\infty(\sigma_N^2 \, R_N)^k
	\,\xrightarrow[]{\,N\to\infty\,}\, \frac{(\hat\beta^2)^K}{1-\hat\beta^2} \,,
\end{split}
\end{equation*}
from which \eqref{eq:fromwhich} follows.

\item\emph{Spectral localization}: given $M,N \in \N$, we define 
disjoint subsets $\bbB_j \subseteq \bbT$ by
\begin{equation*}
	\bbB_j := \big( (N^{\frac{j-1}{M}}, N^{\frac{j}{M}}] \cap \N \big) \times \Z^2
	\qquad \text{for } j=1,\ldots, M \,,
\end{equation*}
and, recalling that $\sigma^2_N (\mathbb{B}_j):= \sum_{A \subset \mathbb{B}_j}q_N(A)^2$,
see \eqref{eq:sigmaB}, we need to show that
\begin{equation*}
\begin{split}
	\lim_{M\to\infty} \ \sum_{j = 1}^M \
	\lim_{N \rightarrow \infty}  \ \sigma^2_N (\mathbb{B}_j) \,=\,
	\sigma^2 \qquad \text{and} \qquad
	\lim_{M \rightarrow \infty} \ \Big\{ \max_{j =
		1, \ldots, M} \ \limsup_{N \rightarrow \infty} \  
		\sigma^2_N (\mathbb{B}_j) \Big\} \,=\, 0 \,.
\end{split}
\end{equation*}
For this it suffices to note that $\sigma^2_N (\mathbb{B}_j) = \bbE[X_{N,M}^{\dom}(j)^2]$
and then to apply \eqref{M1}.
\end{enumerate}
The proof of Theorem~\ref{th:XNGauss} is completed.\qed

\appendix

\section{Some technical results}
\label{sec:technical}

We collect here the proofs of some technical results.

\subsection{Proof of Lemma~\ref{lemmaXdomlow}}
\label{sec:lemmaXdomlow}

We are going to prove that there is a constant $C < \infty$ such that,
for any given $M,K \ \in \N$ and $j \in \{1,\ldots, M\}$, we have
\begin{equation}\label{stimabassoXi}
\begin{split}
	\liminf_{N\to\infty} \ \Xi_{N,M,K}(j) \,\ge\,
	\big(1-(\hat{\beta}^2)^{K}\big) \,
	\int_{\frac{j-1}{M}}^{\frac{j}{M}} 
	\frac{\hat\beta^2 (1- \tfrac{C}{K^2})}
	{1-\hat{\beta}^2(1- \tfrac{C}{K^2}) \, s} \, \dd s \,,
\end{split}
\end{equation}
which clearly implies \eqref{eq:Ij}.

Given $a,b \in \N_0$ as in the range of the sums \eqref{eq:Xi}, we note that for large $N$:
\begin{equation} \label{eq:constr}
	a \le \tfrac{1}{4} K^{-2} b \,.
\end{equation}
This clearly holds if $a=0$, hence for $j=1$,
because $a \le N^{\frac{(j-2)^+}{M}} = 0$, while for $j \ge 2$
from $a \le N^{\frac{j-2}{M}}$ and $b > N^{\frac{j-1}{M}}$
we get $a \le N^{-\frac{1}{M}} b \le \tfrac{1}{4} K^{-2} b$
for large $N$, say $N \ge (2K)^{2M}$.
By \eqref{eq:2momdom}, for fixed $a,b$ and $x$,
the sums over $b' \in [b,Kb]$ and $z,z' \in \Z^2$ in \eqref{eq:Xi} equal
\begin{equation} \label{eq:scrittura}
\begin{split}
	& \sum_{b' \in [b,Kb]} \ \sum_{\substack{|z| \le K \sqrt{b} \\
	|z'| \le K^2 \sqrt{b}}} \
	\bbE \Big[ \big( X^{\rm{dom}}_{N,[a,b_;b']}(x,z;z') 
	\big)^2 \Big] \\
	& \ \,= \sum_{k=1}^\infty (\sigma_N^2)^k  \!
	\sum_{|x_1| \le K \sqrt{b}} q_{b-a}(x_1-x)^2 \!\!\!
	\sum_{\substack{b < n_2 < \ldots < n_{k} \le Kb: \\
	\max\{n_2-b, \ldots, n_k - n_{k-1}\} \le b \\
	x_2,\ldots,x_{k}\in \Z^2: \ |x_k| \le K^2 \sqrt{b}}}
	\prod_{i=2}^k q_{n_i-n_{i-1}}(x_i-x_{i-1})^2 .
\end{split}
\end{equation}
We get a lower bound by keeping
just the first $K$ terms in the sum over $k\in\N$. 
Moreover:
\begin{itemize}
\item we remove the constraint $n_{k} \le Kb$
(because $\max\{n_2-b, \ldots, n_k - n_{k-1}\} \le b$ 
already yields $n_k = b + \sum_{i=2}^k (n_i - n_{i-1}) \le K b$)
and sum freely over the increments
\begin{equation} \label{eq:restr1}
	m_i := n_i - n_{i-1} \in \{1,\ldots, b\} \qquad 
	\text{for } i=2,\ldots, k \,;
\end{equation}

\item we change variables to $y_1 := x_1 - x$ and $y_i  := x_i - x_{i-1}$ for $i \ge 2$,
that we restrict to
\begin{equation*}
	|y_1| \le \tfrac{1}{2} K \sqrt{b-a} \qquad \text{and} \qquad
	|y_i| \le \tfrac{1}{2} K \sqrt{m_i} \qquad \text{for } i \ge 2 \,,
\end{equation*}
which imply both $|x_1| \le K\sqrt{b}$ and $|x_k| \le K^2\sqrt{b}$ as required
by \eqref{eq:scrittura}. Indeed, recalling that $|x| \le K^2 \sqrt{a} \le \frac{1}{2} K \sqrt{b}$
by \eqref{eq:Xi} and \eqref{eq:constr}, we obtain
\begin{gather*}
	|x_1| \le |y_1| + |x| \le \frac{1}{2} K \sqrt{b-a} + \frac{1}{2} K \sqrt{b} \le K\sqrt{b} \,, \\
	|x_k| \le |x_1| + \sum_{i=2}^k |y_i| \le
	K \sqrt{b} + (K-1)\frac{1}{2}K\sqrt{b} \le K^2 \sqrt{b}\,.
\end{gather*}
\end{itemize}
These restrictions yield the following lower bound on \eqref{eq:scrittura}:
\begin{equation}\label{eq:morecom}
	\sum_{k=1}^K (\sigma_N^2)^k \,
	\bigg(\sum_{|y_1| \le \frac{1}{2} K \sqrt{b-a}} q_{b-a}(y_1)^2 \bigg) \,
	\prod_{i=2}^k \bigg( \sum_{m_i=1}^{b} \sum_{|y_i| \le \frac{1}{2} K \sqrt{m_i}} 
	q_{m_i}(y_i)^2 \bigg) \,.
\end{equation}
Recalling that $u_n$ and $R_N$ are defined in \eqref{eq:uN} and \eqref{eq:RN},
we define restricted versions
\begin{equation}\label{eq:uR}
	u^{(K)}_n := \sum_{|y| \le \frac{1}{2} K \sqrt{n}} q_{n}(y)^2 \,, \qquad
	R^{(K)}_N := \sum_{m=1}^N u^{(K)}_m = \sum_{m=1}^N
	\sum_{|y| \le \frac{1}{2} K \sqrt{m}} q_{m}(y)^2  \,,
\end{equation}
so that we can rewrite \eqref{eq:morecom} more compactly as follows:
\begin{equation*}
	\sum_{k=1}^K (\sigma_N^2)^k \,
	u_{b-a}^{(K)} \, \big(R_{b}^{(K)} \big)^{k-1} 
	\,=\, \sigma_N^2 \, u_{b-a}^{(K)} \, \frac{1- \big(\sigma_N^2 R_{b}^{(K)}\big)^K}%
	{1-\sigma_N^2 R_{b}^{(K)}} \,.
\end{equation*}
Bounding $(\sigma_N^2 R_{b}^{(K)})^K \le (\sigma_N^2 R_N)^K$
in the numerator and recalling \eqref{eq:Xi}, we obtain
\begin{equation} \label{eq:boXi}
	\Xi_{N, M,K}(j) \,\ge\,
	\big( 1- \big( \sigma_N^2 R_N \big)^K \big) \,
	\inf_{0 \le a \le  N^{\frac{(j-2)^+}{M}} } \
	\sum_{b \in ( N^{\frac{j-1}{M}} + \log N, \frac{1}{K} N^{\frac{j}{M}} ] } \
	\frac{\sigma_N^2 \, u_{b-a}^{(K)}}{1-\sigma_N^2\, R_{b}^{(K)}} \,,
\end{equation}
where we restricted the sum range to $b \in ( N^{\frac{j-1}{M}} + \log N, \frac{1}{K} N^{\frac{j}{M}} ]$
for later convenience.

We now claim that for some $C < \infty$ we have, for $n, N$ large enough,
\begin{equation} \label{eq:lbuR}
	u_n^{(K)} \ge (1-\tfrac{C}{K^2}) \, \frac{1}{\pi} \, \frac{1}{n} \qquad
	\Longrightarrow \qquad R_N^{(K)} \ge (1-\tfrac{C}{K^2}) \, \frac{1}{\pi} \, \log N \,.
\end{equation}
This follows by \eqref{eq:uR} writing $u_n^{(K)} = u_n - \sum_{|y| > \frac{1}{2}K\sqrt{n}} q_n(y)^2$,
recalling that $u_n \sim \frac{1}{\pi}\frac{1}{n}$ by \eqref{eq:uN}, bounding
$\sup_{y\in\Z^2} q_n(y) \le \frac{c_1}{n}$ by the local limit theorem
(see \eqref{eq:llt} below)
and then estimating
\begin{equation*}
	\sum_{|y| > \frac{1}{2}K\sqrt{n}} q_n(y) = \P(|S_n| > \tfrac{1}{2}K\sqrt{n})
	\le 4 \, \frac{\E[|S_n|^2]}{K^2 \, n} = \frac{4}{K^2} \,.
\end{equation*}
We can plug the bounds \eqref{eq:lbuR} into \eqref{eq:boXi} because,
uniformly for $a,b$ in the sum range, we have
$b \ge b-a \ge \log N \to \infty$ as $N\to\infty$.
Since $\sigma_N^2 \sim \beta_N^2 \sim \pi \hat\beta^2/\log N$,
see \eqref{eq:sub-critical} and \eqref{eq:eta}, for large $N$
we have (possibly enlarging~$C$)
\begin{equation} \label{eq:derhs}
	\frac{\sigma_N^2 \, u_{b-a}^{(K)}}{1-\sigma_N^2 R_{b}^{(K)}} \,\ge\, 
	(1-\tfrac{C}{K^2}) \, \frac{1}{b-a} \, \frac{\frac{\hat\beta^2}{\log N}}{1-
	\, \frac{\hat\beta^2}{\log N} \, (1-\frac{C}{K^2}) \, \log b} \,.
\end{equation}
The r.h.s.\ is a decreasing function of~$b-a$, hence we 
get a lower bound setting $a=0$.
By monotonicity in~$b$, we can then bound the sum in \eqref{eq:boXi}
by an integral:
\begin{equation*}
	\Xi_{N, M,K}(j) \,\ge\,
	(1-\tfrac{C}{K^2})\,
	\big( 1-(\hat\beta^2)^K \big) \,
	\int_{\lceil N^{\frac{j-1}{M}} + \log N \rceil}^{\frac{1}{K} N^{\frac{j}{M}}} \
	\frac{1}{x} \, \frac{\frac{\hat\beta^2}{\log N}}{1-
	\, \frac{\hat\beta^2}{\log N} \, (\log x) \, (1-\frac{C}{K^2})} \, \dd x \,.
\end{equation*}
With the change of variable $x = N^s$, the integral equals
\begin{equation*}
	\int_{a_N}^{b_N} \frac{\hat\beta^2}{1-\hat\beta^2 s \, (1-\frac{C}{K^2})} \, \dd s \qquad
	\text{with} \qquad a_N := \frac{\log \lceil N^{\frac{j-1}{M}} + \log N \rceil }{\log N} \,, 
	\quad b_N := \frac{\log (\frac{1}{K} N^{\frac{j}{M}})}{\log N} \,.
\end{equation*}
Since $\lim_{N\to\infty} a_N = \frac{j-1}{M}$ and $\lim_{N\to\infty} b_N = \frac{j}{M}$,
we have proved \eqref{stimabassoXi}.\qed

\subsection{Proof of Lemma~\ref{th:brackets}}
\label{sec:brackets}

A lower bound for \eqref{eq:thebound} is already provided by \eqref{eq:Ij},
hence it suffices to prove a matching upper bound.
By \eqref{eq:2momdom} with $(a,x)=(0,0)$, we can write
\begin{equation} \label{eq:starting-ub}
\begin{split}
	\sum_{b \le b' \in ( N^{\frac{j-1}{M}}, N^{\frac{j}{M}}]} \
	& \sum_{z, z' \in \Z^2} \ \bbE\big[ 
	X_{N,[0,b;b']}^{\dom}(0,z;z')^2 \big]
	\le \sum_{k=1}^\infty (\sigma_N^2)^k 
	\sum_{b \in ( N^{\frac{j-1}{M}}, N^{\frac{j}{M}}]} \
	\sum_{z\in\Z^2} q_b(z)^2 \\
	& \qquad \times
	\sum_{\substack{b =: n_1 < n_2 < \ldots < n_k < \infty \\
	\max\{ n_2-n_1, \ldots, n_k - n_{k-1} \} \le b}} 
	\ \sum_{\substack{x_1:=z\\ x_2,\ldots,x_{k}\in \Z^2}} 
	\prod_{i=2}^k q_{n_i-n_{i-1}}(x_i-x_{i-1})^2 \,.
\end{split}
\end{equation}
We can sum over the space variables:
by \eqref{eq:uN} and \eqref{eq:RN}, the r.h.s.\ equals
\begin{equation} \label{eq:derhs2}
\begin{split}
	\sum_{k=1}^\infty \, (\sigma_N^2)^k 
	\sum_{b \in ( N^{\frac{j-1}{M}}, N^{\frac{j}{M}}]}
	\, u_{b}  \,
	(R_{b})^{k-1}
	\,=\, 
	\sum_{b \in ( N^{\frac{j-1}{M}}, N^{\frac{j}{M}}]}
	\frac{\sigma_N^2\, u_{b}}{1-\sigma_N^2 \, R_{b}} \,.
\end{split}
\end{equation}
Since $\sigma_N^2 \, u_{b} \sim \frac{\hat\beta^2}{\log N}\frac{1}{b}$ and
$\sigma_N^2 \, R_{b} \sim \frac{\hat\beta^2}{\log N} \log b$, as $N\to\infty$
the r.h.s.\ of \eqref{eq:derhs2} is asymptotic to
\begin{equation} \label{eq:final-ub}
	\begin{split}
	\sum_{b \in ( N^{\frac{j-1}{M}}, N^{\frac{j}{M}}]} 
	\frac{\frac{\hat\beta^2}{\log N} \, \frac{1}{b}}{1- \frac{\hat\beta^2}{\log N} \log b}
	\,\sim\, 
	\int_{N^{\frac{j-1}{M}}}^{N^{\frac{j}{M}}}
	\frac{\frac{\hat\beta^2}{\log N} \, \frac{1}{x}}{1- \frac{\hat\beta^2}{\log N} \log x} \, \dd x 
	\,=\, \int_{\frac{j-1}{M}}^{\frac{j}{M}}
	\frac{\hat\beta^2}{1- \hat\beta^2 \, s\,} \, \dd s\,,
\end{split}
\end{equation}
by the change of variable $x = N^s$. This completes the proof of \eqref{eq:thebound}.\qed

\subsection{Proof of Lemma~\ref{th:clai}}
\label{sec:clai}
We can assume that $j\ge 2$, because if $j=1$ we have $a=0$ and $x=0$,
see \eqref{eq:estXhat}, hence \eqref{eq:estXhatest} trivially holds.

Note that by \eqref{XNdom} we can write
\begin{equation*}
\begin{split}
	& \bbE\big[ X_{N,[a,b;b']}^{\dom}(x,z;z')^2 \big]  
	\,=\, q_{b-a}(z-x)^2 \, F_{N,[b;b']}(z;z') \,, 
\end{split}
\end{equation*}
where we set
\begin{equation*}
	F_{N,[b;b']}(z;z') \,:=\, 
	\sum_{k=1}^\infty (\sigma_N^2)^k 
	\sum_{\substack{b =: n_1 < n_2 < \ldots < n_{k-1} < n_k = b' \\
	1 \le n_2-n_1, \ldots, n_k - n_{k-1} \le b}} 
	\ \sum_{\substack{x_1 := z, \, x_k := z'\\ x_2,\ldots,x_{k-1}\in \Z^2}}  \
	\prod_{i=2}^k q_{n_i-n_{i-1}}(x_i-x_{i-1})^2 \,.
\end{equation*}
The key point is that $F_{N,[b;b']}(z;z')$ does not depend on $(a,x)$.
It follows that
\begin{equation*}
\begin{split}
	\bbE \big[\big(X_{N,[a,b;b']}^{\dom}(x,z;z')- X_{N,[0,b;b']}^{\dom}
	(0,z;z') \big)^2\big]
	= \big( q_{b-a}(z-x) - q_{b}(z) \big)^2 \, F_{N,[b;b']}(z;z') \,,
\end{split}
\end{equation*}
therefore, to prove \eqref{eq:estXhatest},
it is enough to show that for $K,M\in\N$ and $\epsilon > 0$
there is $N_0 = N_0(\epsilon, M, K) < \infty$ such that, for $N > N_0$
and for $a,b,x,z$ as in \eqref{eq:estXhat}, we have
\begin{equation}\label{eq:ggooaall}
	\bigg| 1 - \frac{q_{b}(z)}{q_{b-a}(z-x)} \bigg| \le \epsilon \,.
\end{equation}

We recall the local  limit theorem
\cite[Theorem~2.1.3]{LL10}:
as $n\to\infty$, 
uniformly for $y\in\Z^2$,\footnote{The scaling factor in \eqref{eq:llt} is $n/2$  because the simple
random walk on $\Z^2$ has covariance matrix $\frac{1}{2} I$, while the factor
$2 \, \ind_{(n,y) \in \Z^3_{\mathrm{even}}}$ is due to periodicity.}
\begin{equation} \label{eq:llt}
	q_n(y) = \frac{1}{n/2} \big( g\big(\tfrac{y}{\sqrt{n/2}}\big) + o(1) \big) \,
	2 \, \ind_{(n,y) \in \Z^3_{\mathrm{even}}} \qquad \text{with} \qquad
	g(x) := \frac{\rme^{-|x|^2/2}}{2\pi} \,.
\end{equation}
In particular, for $(n,y) \in \Z^3_{\mathrm{even}}$ in the ``diffusive regime'' we can write
\begin{equation} \label{eq:llt+}
	q_n(y) = \frac{4}{n} \, g\big(\tfrac{y}{\sqrt{n/2}}\big)  \big( 1 + o(1) \big)
	\qquad \text{for } |y| = O(\sqrt{n}) \,.
\end{equation}
Note that $a,b,x,z$ as in \eqref{eq:estXhat} satisfy (recall that $j\ge 2$)
\begin{equation} \label{eq:aba}
	0 \le a \le N^{\frac{j-2}{M}} \le N^{-\frac{1}{M}} b \,, \qquad
	|z| \le K \sqrt{b} \,, \qquad
	|x| \le K^2 \sqrt{a} \le K^2 \sqrt{N^{-\frac{1}{M}}} \, \sqrt{b} \,.
\end{equation}
It follows that for any $K,M\in\N$,
uniformly for  $a,b,x,z$ as in \eqref{eq:estXhat},
we have as $N\to\infty$
\begin{equation*}
	a = o(b) \,, \qquad
	|z| = O(\sqrt{b}) \,, \qquad
	|x| = o(\sqrt{b})  \,,
\end{equation*}
which in turn imply that $|z-x| \le |z| + |x| = O(\sqrt{b}) = O(\sqrt{b-a})$ and hence,
by \eqref{eq:llt+},
\begin{equation*}
	\frac{q_{b}(z)}{q_{b-a}(z-x)}
	= \frac{b-a}{b} \, \exp\bigg(
	\frac{|z-x|^2}{b-a} - \frac{|z|^2}{b} \bigg)
	(1+o(1)) 
	\,\xrightarrow[\,N\to\infty\,]{}\, 1 \,.
\end{equation*}
This completes the proof of \eqref{eq:ggooaall}, hence of \eqref{eq:estXhatest}.\qed

\subsection{Proof of \eqref{M2}}
\label{sec:M2}

The random variables $\eta_N$ in \eqref{eq:eta}
satisfy $\sup_N \bbE[|\eta_N|^{\overline{p}}] < \infty$
for all $\overline{p} < \infty$, by the assumption \eqref{eq:disorder}
(see \cite[eq. (6.7)]{CSZ17a}). We can then estimate $\bbE \big[ |X_{N,M}^{\dom}(j)|^p\big]^{\frac{2}{p}}$
by the hypercontractive bound \eqref{eq:hypercontractivity2}, which gives rise to the
r.h.s.\ of \eqref{eq:starting-ub} with $\sigma_N^2$ replaced by $C_p \, \sigma_N^2$.
We can then follow the proof of Lemma~\ref{th:brackets}
in Appendix~\ref{sec:brackets} verbatim though
\eqref{eq:derhs2} and \eqref{eq:final-ub}, where we note that the
replacement of $\sigma_N^2$ by $C_p \, \sigma_N^2$
amounts to replace $\hat\beta^2$ by $C_p \, \hat\beta^2$, by \eqref{eq:eta} and \eqref{eq:sub-critical}.
Since $\hat\beta < 1$ and $\lim_{p \downarrow 2} C_p = 1$, 
see \cite[Theorem~B.1]{CSZ20},
we can fix $p_{\hat\beta} > 2$ and $\tilde c = \tilde c_{\hat\beta} < 1$ such that
for all $2 < p \le p_{\hat\beta}$ we can bound
$C_p \hat\beta^2 \le \tilde c < 1$, hence
\begin{equation}\label{M1bis}
	\limsup_{N\to\infty}
	\ \bbE \big[ |X_{N,M}^{\dom}(j)|^p\big]^{\frac{2}{p}}
	\,\le\, \int_{\frac{j-1}{M}}^{\frac{j}{M}} \frac{C_p\hat{\beta}^2}{1-C_p\hat{\beta}^2s} 
	\, \text{d}s
	\,\le\, \frac{\tilde c/(1-\tilde c)}{M} \,, 
\end{equation}
which completes the proof.\qed


\begin{thebibliography}{NouPec12}

\bibitem[BC98]{BC98}
L. Bertini and N. Cancrini.
The two-dimensional stochastic heat equation: renormalizing a multiplicative noise.
{\em J. Phys. A: Math. Gen.} 31 (1998), 615-622.

\bibitem[Bil95]{Bil95}
P. Billingsley.
Probability and Measure. Third Edition (1995), John Wiley and Sons.

\bibitem[CSZ17a]{CSZ17a}
F. Caravenna, R. Sun, N. Zygouras.
Polynomial chaos and scaling limits of disordered systems.
{\em J. Eur. Math. Soc.} 19 (2017), 1--65.

\bibitem[CSZ17b]{CSZ17b}
F.~Caravenna, R.~Sun,  N.~Zygouras.
Universality in marginally relevant disordered systems.
{\em Ann. Appl. Probab.}  27 (2017), 3050--3112.

\bibitem[CSZ19b]{CSZ19b}
F. Caravenna, R. Sun, N. Zygouras.
On the moments of the (2+1)-dimensional directed polymer and 
stochastic heat equation in the critical window.
{\em Commun. Math. Phys.} 372 (2019), 385-440.

\bibitem[CSZ20]{CSZ20}
F. Caravenna, R. Sun, N. Zygouras.
The two-dimensional KPZ equation in the entire subcritical regime.
{\em Ann. Probab.} 48 (2020), 1086--1127.

\bibitem[CSZ21+]{CSZ21+}
F. Caravenna, R. Sun, N. Zygouras.
The Critical 2d Stochastic Heat Flow.
arXiv:2109.03766 (2021).


\bibitem[Com17]{C17}
F. Comets.
{\em Directed Polymers in Random Environments}.
Lecture Notes in Mathematics, 2175. Springer, Cham, 2017.

\bibitem[CD20]{CD20}
S. Chatterjee, A. Dunlap. Constructing a solution of the (2+1)-dimensional KPZ equation. 
{\em Ann. Probab.} 48 (2020), 1014--1055.

\bibitem[CCM20]{CCM20}
F. Comets, C. Cosco, C. Mukherjee. Renormalizing the Kardar-Parisi-Zhang equation in $d\ge 3$ 
in weak disorder. {\em J. Stat. Phys.} 179 (2020) 713--728.

\bibitem[CCM21+]{CCM21+}
F. Comets, C. Cosco, C. Mukherjee. Space-time fluctuation of the Kardar-Parisi-Zhang 
equation in $d\ge 3$ and the Gaussian free field. arXiv:1905.03200 (2021).

\bibitem[CNN20+]{CNN20+}
C. Cosco, S. Nakajima, M. Nakashima. Law of large numbers and fluctuations in the sub-critical 
and $L^2$ regions for SHE and KPZ equation in dimension $d\ge 3$. arXiv:2005.12689 (2020).

\bibitem[CZ21+]{CZ21+}
C. Cosco, O. Zeitouni.
Moments of partition functions of 2D Gaussian polymers in the weak disorder regime.
arXiv:2112.03767 (2021).

\bibitem[Cot23]{cf:Cottini}
F. Cottini. {\em Ph.D. Thesis}. In preparation (2023).

\bibitem[DG20+]{DG20+}
A. Dunlap, Y. Gu.
A forward-backward SDE from the 2D nonlinear stochastic heat equation.
arXiv:2010.03541 (2020).

\bibitem[DGRZ20]{DGRZ20}
A. Dunlap, Y. Gu, L. Ryzhik, O. Zeitouni.
Fluctuations of the solutions to the KPZ equation in dimensions three and higher.
{\em Probab. Theory Relat. Fields} 176 (2020), 1217--1258.

\bibitem[deJ87]{dJ87}
P. de Jong.
A central limit theorem for generalized quadratic forms.
{\em Probab. Theory Relat. Fields} 75 (1987), 261--277.

\bibitem[deJ90]{dJ90}
P. de Jong.
A central limit theorem for generalized multilinear forms.
{\em J. Multivariate Anal.} 34 (1990),  275--289.

\bibitem[G20]{G20}
Y. Gu. Gaussian fluctuations from the 2D KPZ equation. 
{\em Stoch. Partial Differ. Equ. Anal. Comput.} 8 (2020), 150--185.

\bibitem[GQT21]{GQT21}
Y. Gu, J. Quastel, L.-C. Tsai.
Moments of the 2D SHE at criticality.
{\em Prob. Math. Phys.} 2 (2021), 179-219.

\bibitem[GRZ18]{GRZ18}
Y. Gu, L. Ryzhik, O. Zeitouni.
The Edwards-Wilkinson limit of the random heat equation in dimensions three and higher. 
{\em Comm. Math. Phys.} (2018) 363 (2), 351--388.

\bibitem[Ito51]{Ito51} 
K. Ito. Multiple Wiener Integral.
J. Math. Society of Japan, Vol.\ 3 (1951), N.\ 1, 157--169.


\bibitem[Jan97]{J97}
S.~Janson.
Gaussian Hilbert spaces.
Cambridge Tracts in Mathematics, Vol.~129.
{\em Cambridge University Press}, Cambridge (1997).

\bibitem[LL10]{LL10}
G.F. Lawler, V. Limic. Random walk: a modern introduction. Cambridge University Press (2010).

\bibitem[LZ20+]{LZ20+}
D. Lygkonis, N. Zygouras.
Edwards-Wilkinson fluctuations for the directed polymer in the full $L^2$-regime for dimensions
$d\ge 3$.
{\it arXiv.org}: 2005.12706 (2020), Ann. Inst. H. Poincar\'e Probab. Statist. (to appear).

\bibitem[LZ21+]{LZ21+}
D. Lygkonis, N. Zygouras.
Moments of the 2d directed polymer in the subcritical regime and 
a generalisation of the Erdös-Taylor theorem.
{\it arXiv.org}: 2109.06115 (2021).

\bibitem[MU18]{MU18}
J. Magnen, J. Unterberger.
The scaling limit of the KPZ equation in space dimension 3 and higher.
{\em J. Stat. Phys.} 171 (2018), Volume 171, 543--598.

\bibitem[MOO10]{MOO10}
E. Mossel, R. O'Donnell, K. Oleszkiewicz.
Noise stability of functions with low influences: Invariance and optimality.
{\it Ann. Math} 171 (2010), 295--341.

\bibitem[NN21+]{NN21+}
S. Nakajima, M. Nakashima. Fluctuations of two-dimensional stochastic heat equation and 
KPZ equation in subcritical regime for general initial conditions. arXiv:2103.07243 (2021).

\bibitem[NuaPec05]{NP05}
D. Nualart and G. Peccati.
Central limit theorems for sequences of multiple stochastic integrals.
{\it Ann. Probab.} 33 (2005), 177--193.

\bibitem[NouPec12]{NP12}
I. Nourdin and G. Peccati.
Normal Approximations with Malliavin Calculus. From Stein's Method to Universality.
{\it Cambridge University Press} (2012).

\bibitem[NPR10]{NPR10}
I. Nourdin, G. Peccati, G. Reinert.
Invariance principles for homogeneous sums: universality of Gaussian Wiener chaos.
{\it Ann. Probab.} 38 (2010), 1947--1985.

\bibitem[Rot79]{R79}
V.I. Rotar’.
Limit theorems for polylinear forms.
{\em J. Multivariate Anal.} 9 (1979), 511--530.
\end{thebibliography}
\end{document}